\documentclass[reqno,letterpaper]{gthack} 
\usepackage[margin=1.25in, left=1.5in, right=1.3in]{geometry}
\usepackage{multirow}

\usepackage[numbers]{natbib}
\setcitestyle{open={[},close={]}}

\usepackage{sidecap}
\usepackage{hyperref}

\usepackage{pinlabel}

\titleformat{\section}{\large\bfseries}{}{0pt}{\center \thesection.  }
\titlespacing{\section}{0pt}{*0}{*3}
\titleformat{\subsection}[runin]{\bfseries}{}{0pt}{\thesubsection \ \  }
\titlespacing{\subsection}{0pt}{*2}{*1.65}

\titleformat{\subsubsection}[runin]{\it}{}{0pt}{\thesubsubsection \ \  }
\titlespacing{\subsubsection}{0pt}{*2}{*1.65}

\usepackage{xfrac}

\usepackage{graphicx}
\usepackage{enumitem}
\usepackage{tikz-cd}
\allowdisplaybreaks
\usepackage{float}
\usepackage{array}

\usepackage{comment}
\usepackage[margin=.5cm,font=small,labelfont=small]{caption}

\title{Khovanov homology and exotic surfaces in the 4-ball}
\author{Kyle Hayden} \address{Rutgers University, Newark, NJ 07102, USA} 
\email{kyle.hayden@rutgers.edu}

\author{Isaac Sundberg} 
\address{Max Planck Institut f\"ur Mathematik, 53111 Bonn, Germany} 
\email{sundberg@mpim-bonn.mpg.de}

\theoremstyle{plain}
\newtheorem{thm}{Theorem}[section]   \newtheorem{lem}[thm]{Lemma}
              \newtheorem{prop}[thm]{Proposition}
\newtheorem{cor}[thm]{Corollary}

\newtheorem{heuristic}[thm]{Heuristic}
\newtheorem{big-ex}[thm]{Example}

\newtheorem*{ques*}{Question}

\newtheorem*{conj*}{Conjecture}
\newtheorem*{thmcusp*}{Theorem B}
\newtheorem*{thm-plain}{Theorem}

\theoremstyle{definition}
\newtheorem{ex}[thm]{Example}
\newtheorem{rem}[thm]{Remark}       \newtheorem*{rem*}{Remark}             
\theoremstyle{remark}
\newtheorem{ex-main}[thm]{Example}

\newcommand{\cc}{\mathbb{C}}
\newcommand{\rr}{\mathbb{R}}
\newcommand{\zz}{\mathbb{Z}}

\newcommand{\st}{{\mathrm{st}}}
\newcommand{\id}{\operatorname{id}}
\newcommand{\Kh}{\operatorname{Kh}}

\newcommand{\CKh}{\operatorname{\mathcal{C}Kh}}

\newcommand{\Diff}{\operatorname{Diff}}
\newcommand{\Sym}{\operatorname{Sym}}

\renewcommand{\S}{\textsection}

\newcommand{\hfk}{\widehat{HFK}}

\newcommand{\crossing}{\raisebox{-.2\height}{\includegraphics[scale=.2]{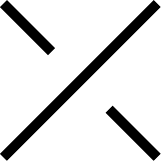}} \null}
\newcommand{\zsmooth}{\raisebox{-.1\height}{\includegraphics[scale=.2]{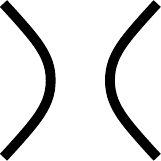}}\null}
\newcommand{\osmooth}{\raisebox{-.1\height}{\includegraphics[scale=.2]{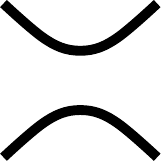}}\null}
\newcommand{\kcap}{\includegraphics[scale=.2,trim=0 .2cm 0 0]{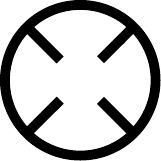}\null}

\newcommand{\kcup}{\includegraphics[scale=.15,trim=0 .2cm 0 0]{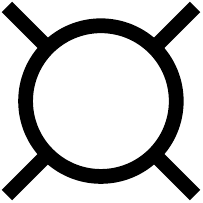}\null}

\newcommand{\ksmooth}{\includegraphics[scale=.2,trim=0 .2cm 0 0]{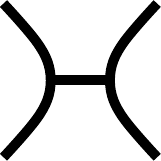}\null}
\newcommand{\kdot}{\includegraphics[scale=.2,trim=0 .2cm 0 0]{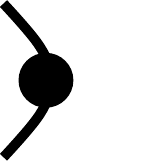}\null}

\newcommand{\kocirc}{\raisebox{-.375\height}{\includegraphics[scale=.225,trim=0 .2cm 0 0]{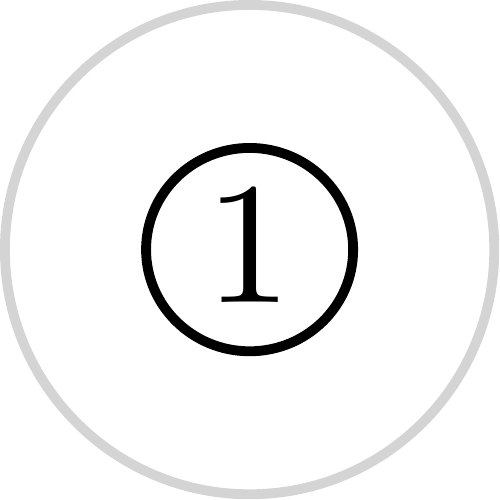}}\null}
\newcommand{\kxcirc}{\raisebox{-.375\height}{\includegraphics[scale=.225,trim=0 .2cm 0 0]{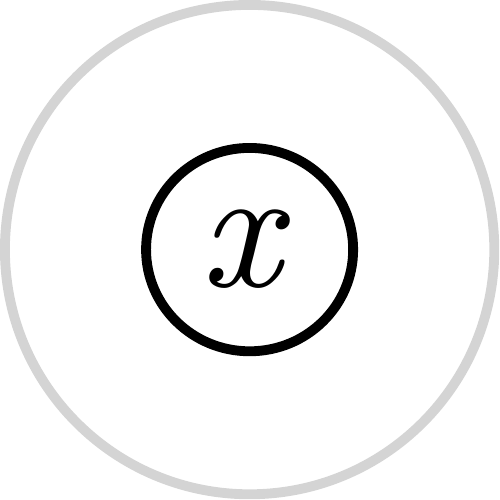}}\null}

\newcommand{\kcobdo}{\raisebox{-.375\height}{\includegraphics[scale=.225,trim=0 .2cm 0 0]{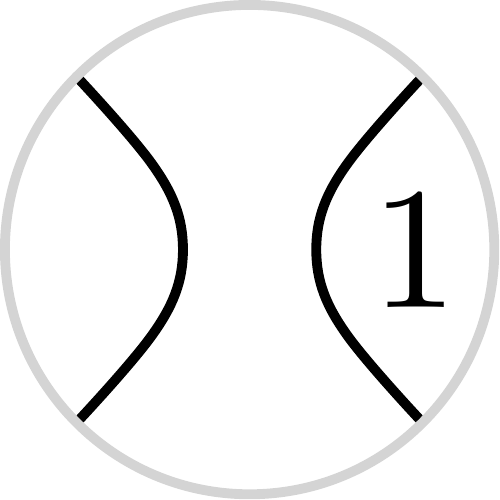}}\null}
\newcommand{\kcobdx}{\raisebox{-.375\height}{\includegraphics[scale=.225,trim=0 .2cm 0 0]{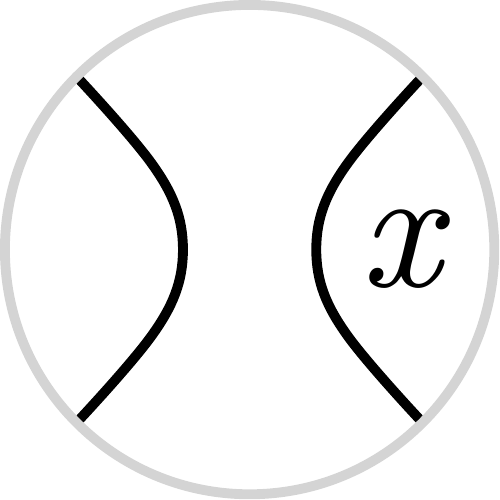}}\null}
\newcommand{\kcobdoo}{\raisebox{-.375\height}{\includegraphics[scale=.225,trim=0 .2cm 0 0]{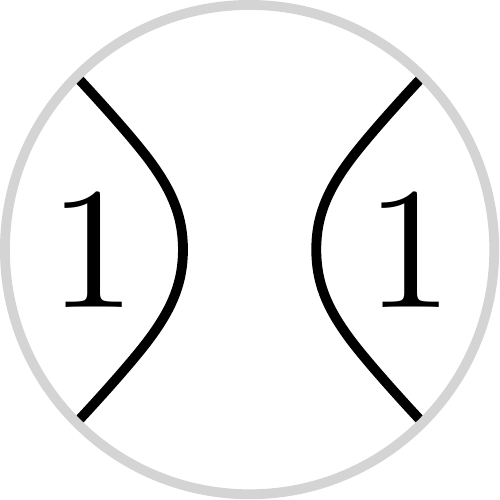}}\null}
\newcommand{\kcobdox}{\raisebox{-.375\height}{\includegraphics[scale=.225,trim=0 .2cm 0 0]{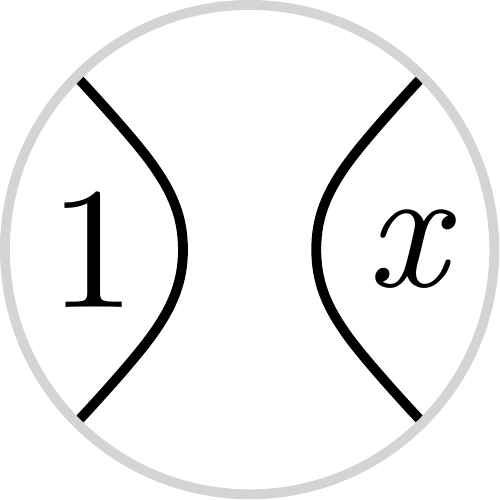}}\null}
\newcommand{\kcobdxo}{\raisebox{-.375\height}{\includegraphics[scale=.225,trim=0 .2cm 0 0]{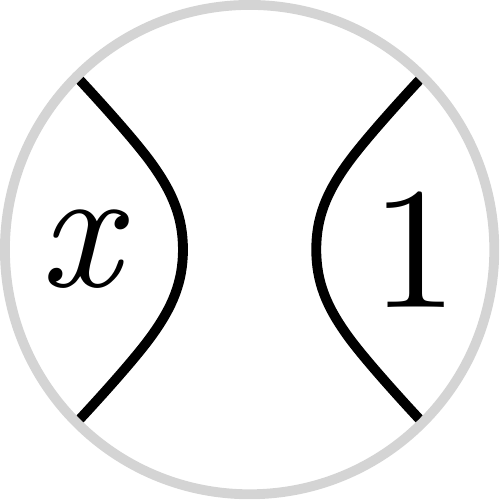}}\null}
\newcommand{\kcobdxx}{\raisebox{-.375\height}{\includegraphics[scale=.225,trim=0 .2cm 0 0]{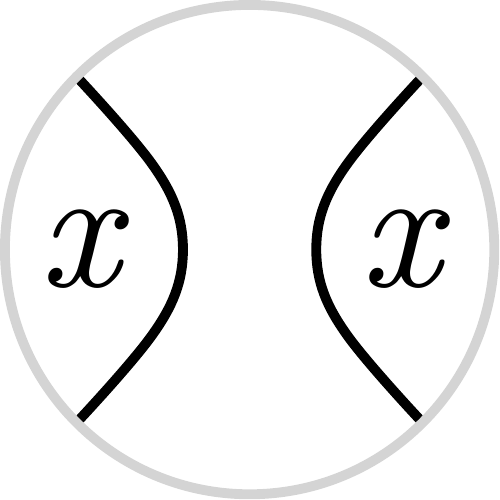}}\null}

\newcommand{\kcobio}{\raisebox{-.375\height}{\includegraphics[scale=.225,trim=0 .2cm 0 0]{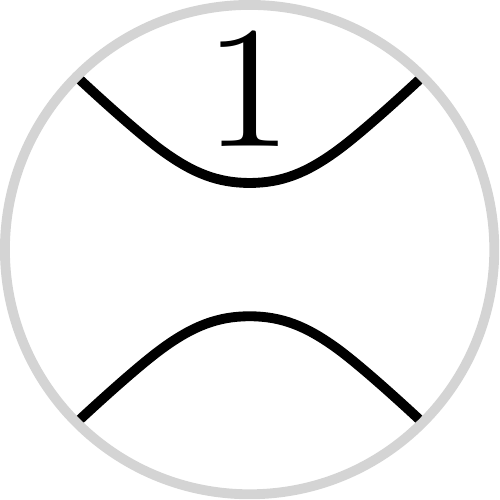}}\null}
\newcommand{\kcobix}{\raisebox{-.375\height}{\includegraphics[scale=.225,trim=0 .2cm 0 0]{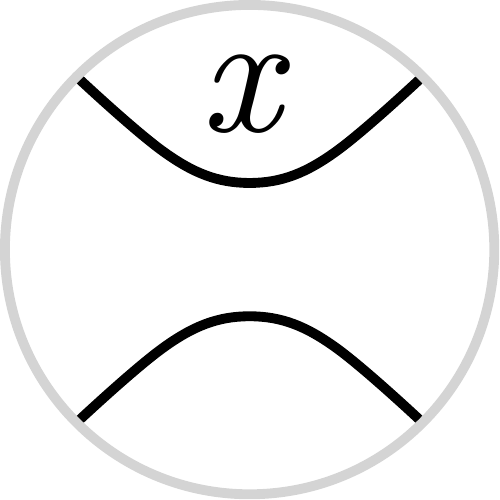}}\null}

\newcommand{\kcobiox}{\raisebox{-.375\height}{\includegraphics[scale=.225,trim=0 .2cm 0 0]{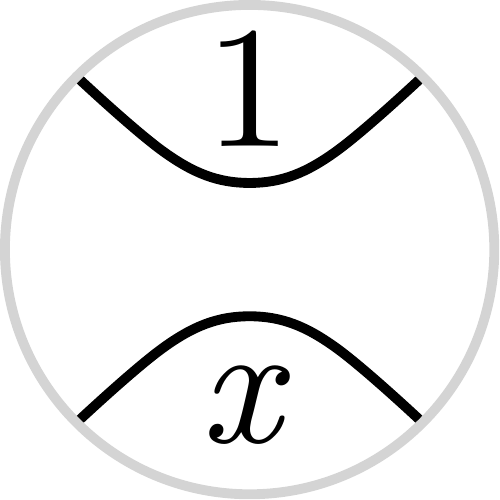}}\null}
\newcommand{\kcobixo}{\raisebox{-.375\height}{\includegraphics[scale=.225,trim=0 .2cm 0 0]{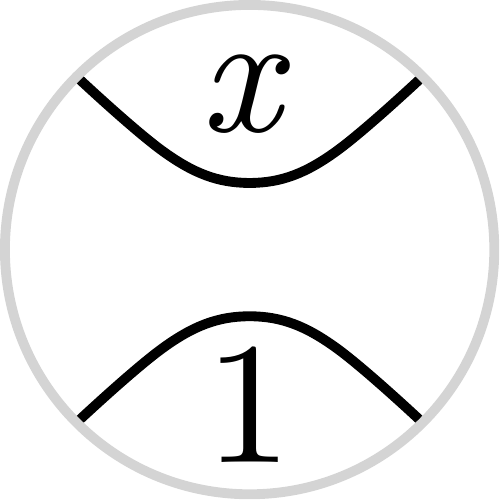}}\null}
\newcommand{\kcobixx}{\raisebox{-.375\height}{\includegraphics[scale=.225,trim=0 .2cm 0 0]{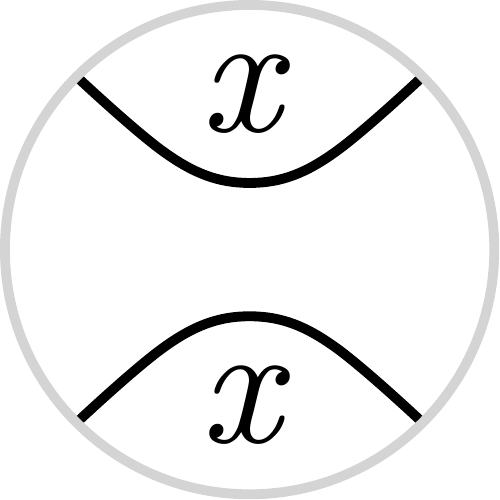}}\null}

\newcommand{\kropos}{\raisebox{-.375\height}{\includegraphics[scale=.225,trim=0 .2cm 0 0]{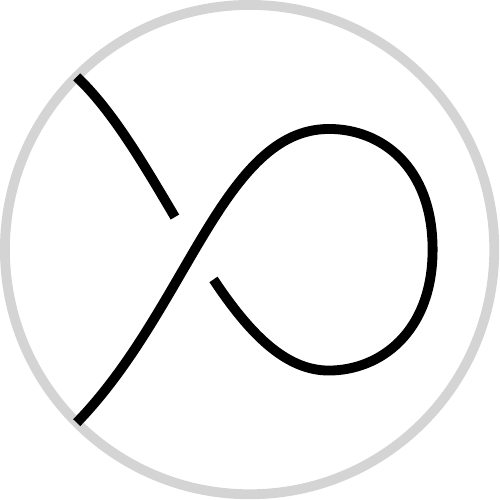}}\null}
\newcommand{\kroneg}{\raisebox{-.375\height}{\includegraphics[scale=.225,trim=0 .2cm 0 0]{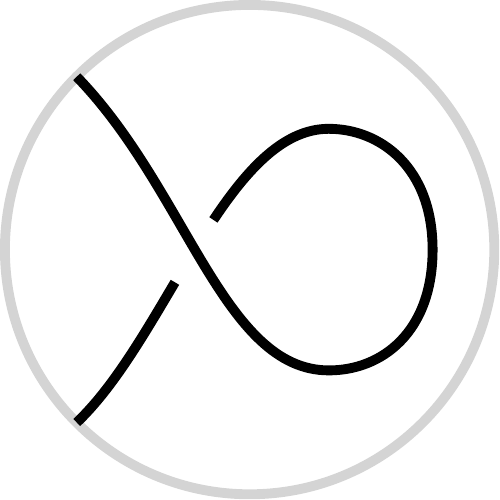}}\null}
\newcommand{\kroarc}{\raisebox{-.375\height}{\includegraphics[scale=.225,trim=0 .2cm 0 0]{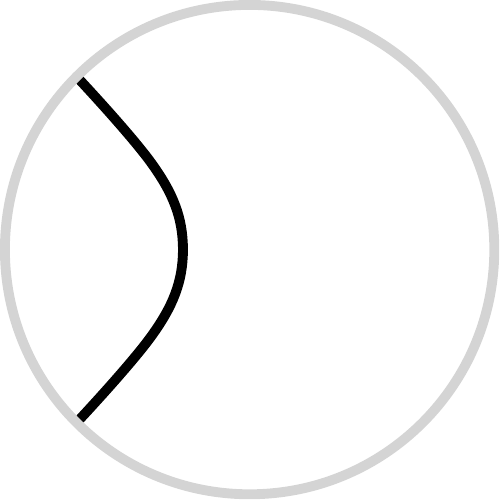}}\null}
\newcommand{\krosa}{\raisebox{-.375\height}{\includegraphics[scale=.225,trim=0 .2cm 0 0]{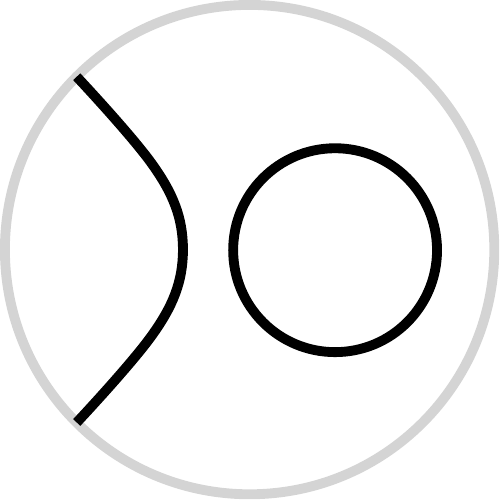}}\null}
\newcommand{\krosb}{\raisebox{-.375\height}{\includegraphics[scale=.225,trim=0 .2cm 0 0]{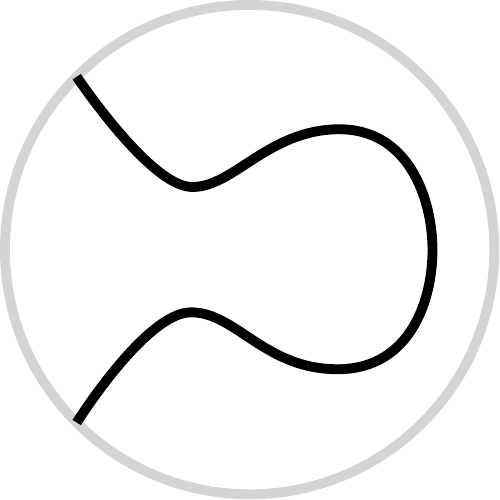}}\null}
\newcommand{\kroma}{\raisebox{-.375\height}{\includegraphics[scale=.225,trim=0 .2cm 0 0]{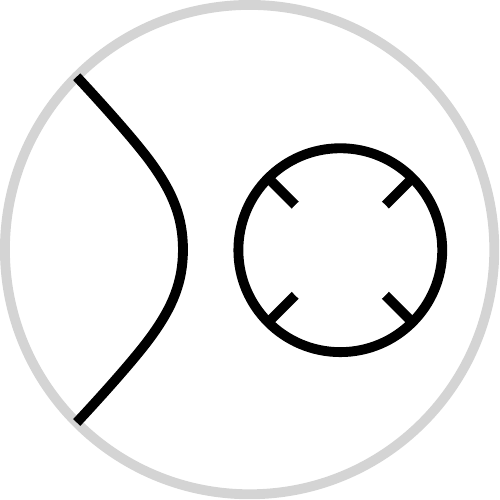}}\null}
\newcommand{\kromb}{\raisebox{-.375\height}{\includegraphics[scale=.225,trim=0 .2cm 0 0]{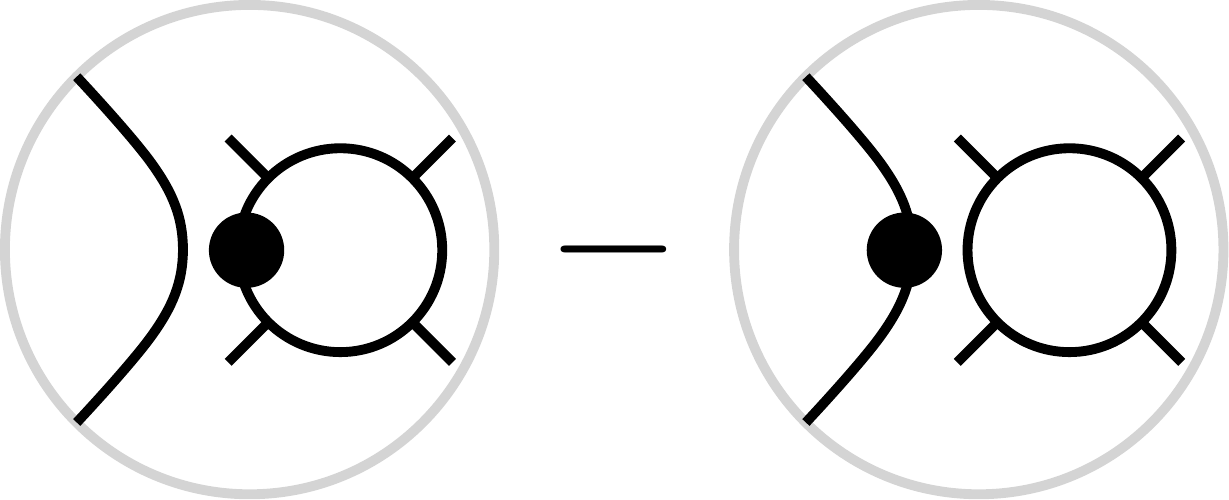}}\null}
\newcommand{\kromc}{\raisebox{-.375\height}{\includegraphics[scale=.225,trim=0 .2cm 0 0]{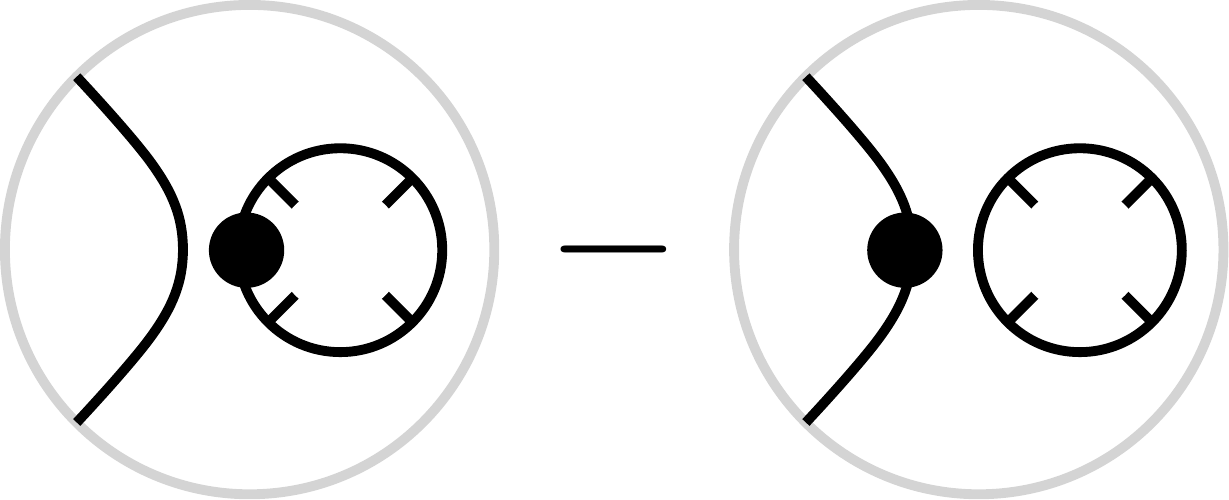}}\null}
\newcommand{\kromd}{\raisebox{-.375\height}{\includegraphics[scale=.225,trim=0 .2cm 0 0]{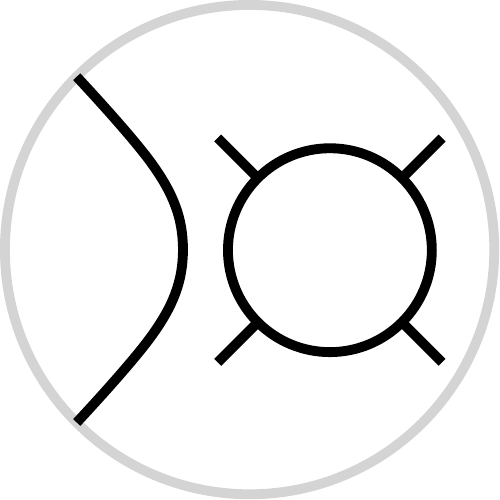}}\null}

\newcommand{\krtcr}{\raisebox{-.375\height}{\includegraphics[scale=.225,trim=0 .2cm 0 0]{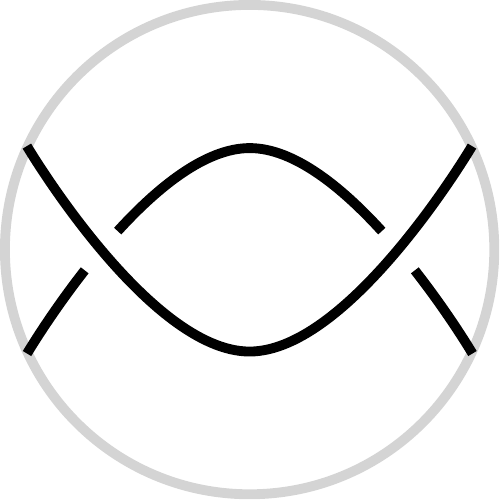}}\null}
\newcommand{\krtcrl}{\raisebox{-.375\height}{\includegraphics[scale=.225,trim=0 .2cm 0 0]{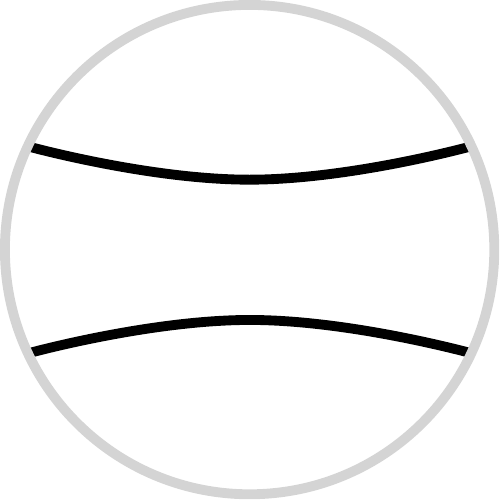}}\null}
\newcommand{\krtsa}{\raisebox{-.375\height}{\includegraphics[scale=.225,trim=0 .2cm 0 0]{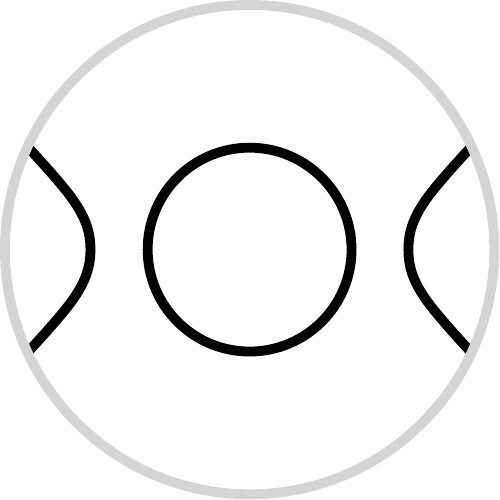}}\null}
\newcommand{\krtsb}{\raisebox{-.375\height}{\includegraphics[scale=.225,trim=0 .2cm 0 0]{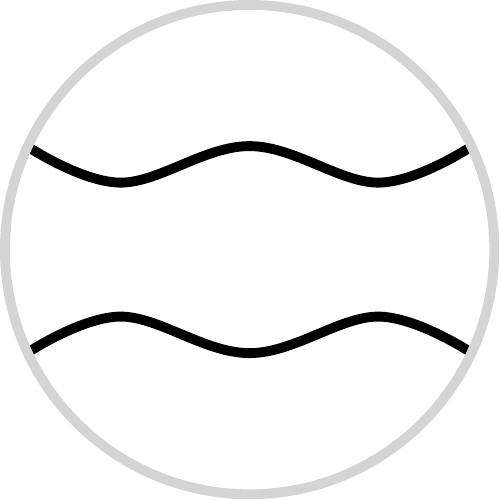}}\null}
\newcommand{\krtsc}{\raisebox{-.375\height}{\includegraphics[scale=.225,trim=0 .2cm 0 0]{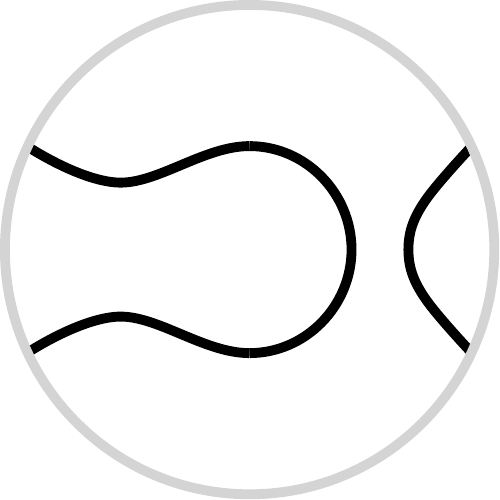}}\null}
\newcommand{\krtsd}{\raisebox{-.375\height}{\includegraphics[scale=.225,trim=0 .2cm 0 0]{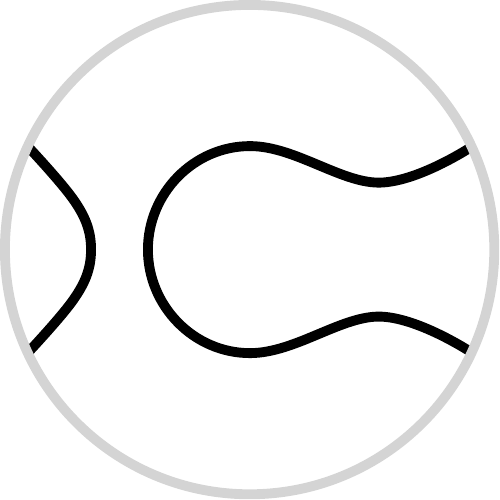}}\null}
\newcommand{\krtma}{\raisebox{-.375\height}{\includegraphics[scale=.225,trim=0 .2cm 0 0]{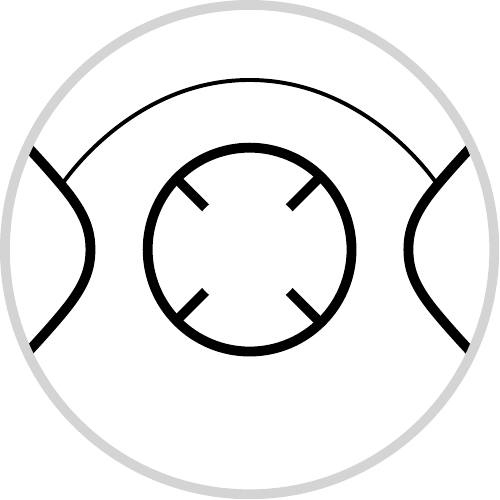}}\null}
\newcommand{\krtmb}{\raisebox{-.375\height}{\includegraphics[scale=.225,trim=0 .2cm 0 0]{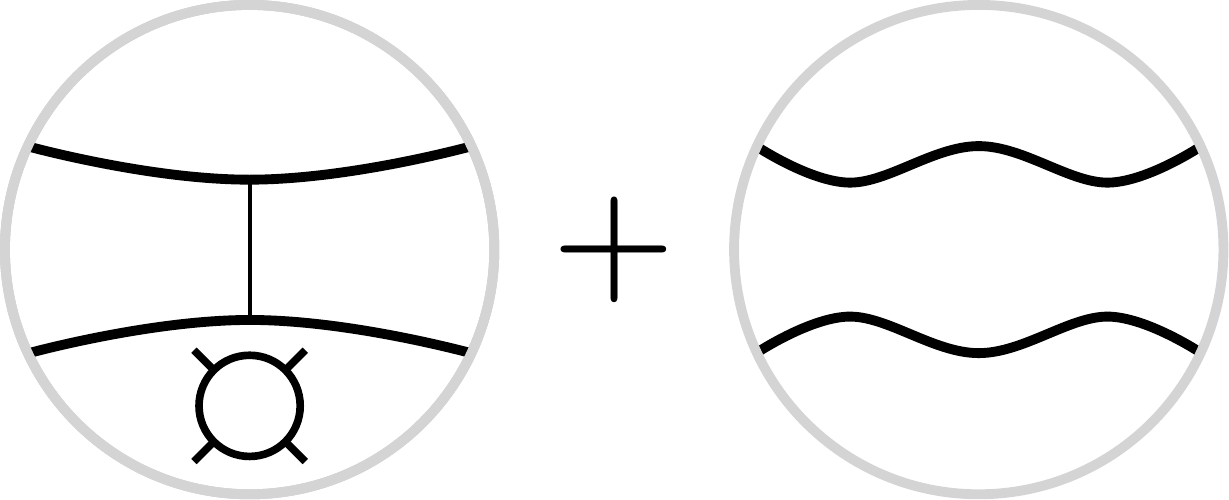}}\null}

\newcommand{\zdott}{\raisebox{-.2\height}{\includegraphics[scale=.215]{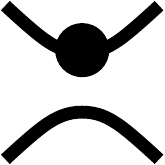}}\null}
\newcommand{\zdotb}{\raisebox{-.2\height}{\includegraphics[scale=.215]{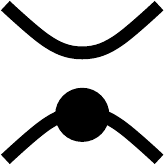}}\null}

\newcommand{\cycle}{\phi}
\newcommand{\class}{\phi}
\newcommand{\secondcycle}{\theta} 

\begin{document}

\begin{abstract}\vspace{-.075in}

We show that the cobordism maps on Khovanov homology can distinguish smooth surfaces in the 4-ball that are exotically knotted (i.e., isotopic through ambient homeomorphisms but not ambient diffeomorphisms). We develop new techniques for distinguishing cobordism maps on Khovanov homology, drawing on knot symmetries and braid factorizations. We also show that Plamenevskaya's transverse invariant in Khovanov homology is preserved by maps induced by positive ascending cobordisms.
\end{abstract}

\maketitle

\vspace{-.45in}

\section{Introduction} 
\null\vspace{-25pt}
Much of the power of Khovanov homology \cite{khovanov} is derived from its functoriality under link cobordisms. That is, an oriented  cobordism $\Sigma: L_0 \to L_1$ between links in $S^3$ induces a 
cobordism map $\Kh(\Sigma): \Kh(L_0) \to \Kh(L_1)$, which is well-defined up to sign and invariant up to isotopy of $\Sigma$ rel boundary \cite{jacobsson}. 

This functoriality is key to the growing number of 4-dimensional applications of Khovanov homology.  For example, it is used to prove that Rasmussen's $s$-invariant of a knot \cite{rasmussen:s} (defined using Lee's deformed theory \cite{lee}) 
gives a lower bound on (twice) the minimal genus of any smooth, orientable surface the knot bounds in $B^4$.  Rasmussen's invariant, in turn, is key to several spectacular applications of Khovanov homology,  such as Piccirillo's proof that the Conway knot is not slice \cite{picc:conway}, Rasmussen's  reproofs of the Milnor conjecture \cite{rasmussen:s} and the existence of exotic smooth structures on $\rr^4$ \cite{rasmussen:polynomials}, and Lambert-Cole's reproof of the adjunction inequality for surfaces in symplectic 4-manifolds via trisections and contact geometry \cite{plc:adjunction}. To date,  results using Rasmussen-type invariants appear to be the only known applications of Khovanov homology to the detection of 
\emph{exotic} phenomena ---  differences between the smooth and topological categories in dimension four.  
It remains a major goal to use Khovanov homology and its generalizations to shed new light on 4-manifolds and the exotic phenomena they exhibit \cite{manandmachine, morrison-walker, morrison-walker-wedrich, manolescu-neithalath,mmsw:s,mww:skein}.

In this paper, we show that Khovanov homology  can directly distinguish between \emph{exotically knotted} surfaces in $B^4$, i.e., pairs of smooth surfaces that are topologically isotopic through ambient homeomorphisms but not 
 ambient diffeomorphisms. This provides a direct, elementary,  combinatorial approach to distinguishing exotic surfaces. 

\begin{thm}\label{thm:main}
For all integers $g \geq 0$, there are infinitely many knots $K \subset S^3$ that each bound a pair of smooth, orientable, genus-$g$ surfaces   $\Sigma,\Sigma' \subset B^4$  that are topologically isotopic rel~boundary yet induce different maps $\Kh(\Sigma) \neq \pm \Kh(\Sigma')$, hence are not smoothly isotopic rel boundary. Moreover, $K$ can be chosen to have trivial symmetry group, implying there is no smooth isotopy of $B^4$ carrying $\Sigma$ to $\Sigma'$.
\end{thm}

The surfaces $\Sigma$ and $\Sigma'$ are modeled on a core pair of examples, drawn from \cite{hayden:curves} and depicted in Figure~\ref{fig:main-disks}. We distinguish their induced maps in \S\ref{subsec:core}. 

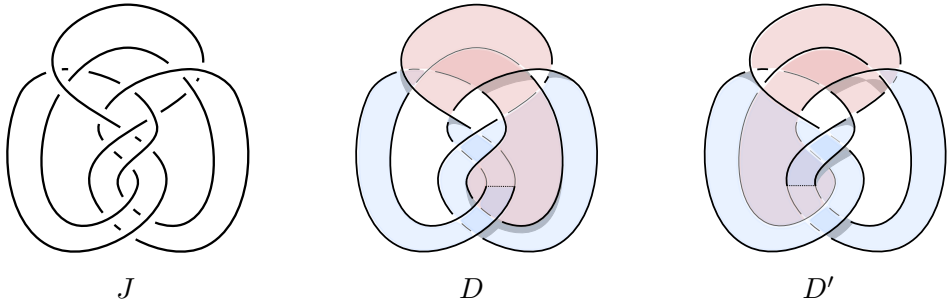
\begin{figure}\center
	\def\svgwidth{.87\linewidth}
\begingroup%
  \makeatletter%
  \providecommand\color[2][]{%
    \errmessage{(Inkscape) Color is used for the text in Inkscape, but the package 'color.sty' is not loaded}%
    \renewcommand\color[2][]{}%
  }%
  \providecommand\transparent[1]{%
    \errmessage{(Inkscape) Transparency is used (non-zero) for the text in Inkscape, but the package 'transparent.sty' is not loaded}%
    \renewcommand\transparent[1]{}%
  }%
  \providecommand\rotatebox[2]{#2}%
  \newcommand*\fsize{\dimexpr\f@size pt\relax}%
  \newcommand*\lineheight[1]{\fontsize{\fsize}{#1\fsize}\selectfont}%
  \ifx\svgwidth\undefined%
    \setlength{\unitlength}{759.03727854bp}%
    \ifx\svgscale\undefined%
      \relax%
    \else%
      \setlength{\unitlength}{\unitlength * \real{\svgscale}}%
    \fi%
  \else%
    \setlength{\unitlength}{\svgwidth}%
  \fi%
  \global\let\svgwidth\undefined%
  \global\let\svgscale\undefined%
  \makeatother%
  \begin{picture}(1,0.31817103)%
    \lineheight{1}%
    \setlength\tabcolsep{0pt}%
    \put(0,0){\includegraphics[width=\unitlength,page=1]{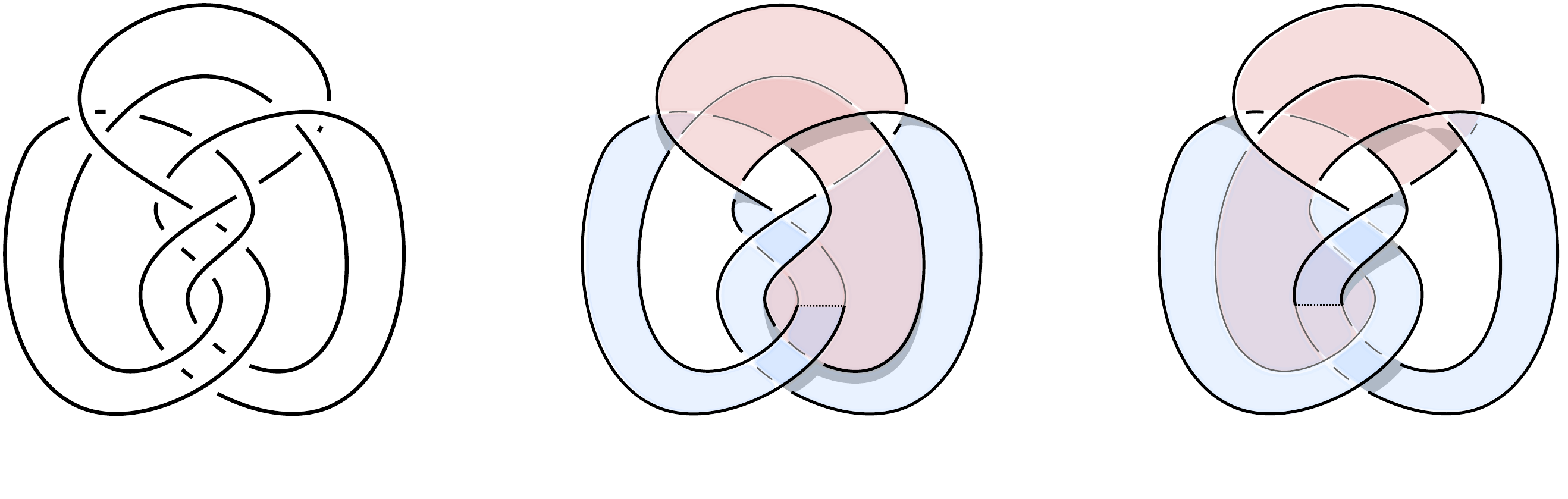}}%
    \put(0.11586816,0.00499062){\makebox(0,0)[lt]{\lineheight{1.25}\smash{\begin{tabular}[t]{l}$J$\end{tabular}}}}%
    \put(0.47776745,0.00499062){\makebox(0,0)[lt]{\lineheight{1.25}\smash{\begin{tabular}[t]{l}$D$\end{tabular}}}}%
    \put(0.8403117,0.00499062){\makebox(0,0)[lt]{\lineheight{1.25}\smash{\begin{tabular}[t]{l}$D'$\end{tabular}}}}%
  \end{picture}%
\endgroup%

	\caption{The slice disks $D$ and $D'$ bounded by the knot $J$ are topologically isotopic rel boundary yet induce distinct cobordism maps on Khovanov homology.}\label{fig:main-disks}
\end{figure}

\begin{rem}\label{remark:akbulut}
	Surprisingly, the disks in Figure~\ref{fig:main-disks} arise separately in \cite{akbulut:zeeman}, where they were distinguished via Donaldson theory. These pairs of disks will be shown to be equivalent in \cite{choi-hayden}.
	Khovanov previously asked if these disks 
	can be distinguished by their induced maps; this paper gives an affirmative answer to this question.
\end{rem}

\vspace{-2.5pt}

Our strategy is simple and direct: viewing the surfaces $\Sigma$ and $\Sigma'$ as link cobordisms $K \to \emptyset$, we distinguish the induced maps $\Kh(\Sigma)$ and $\Kh(\Sigma')$ by finding an explicit homology class $\class \in \Kh(K)$ such that $\Kh(\Sigma)(\class) \neq 0$ yet $\Kh(\Sigma')(\class)=0$. 
 This approach is dual to that of \cite{sundberg-swann}, where instead the surfaces are viewed as cobordisms $\emptyset \to K$ and the induced maps are distinguished by their associated \textit{relative Khovanov-Jacobsson classes}, 
  i.e., the classes in $\Kh(K)$ to which they map the generator of $\Kh(\emptyset) = \zz$, modulo sign. These classes are convenient to define but can be impractical as  explicit obstructions, requiring significant computational endurance to both calculate and distinguish. In contrast,  viewing the surfaces as link cobordisms $K \to \emptyset$, we  limit the computational complexity by choosing the class $\class \in \Kh(K)$ and can easily compare the integers $\Kh(\Sigma)(\class)$ and $\Kh(\Sigma')(\class)$. (Formally, these two approaches reflect the duality of Khovanov homology under mirroring \cite[\S7.3]{khovanov}.)

This shift to the dual perspective comes at a cost, as we must directly identify classes  $\phi \in \Kh(K)$ that distinguish the surfaces bounded by $K$. 
We use two perspectives to help illuminate classes in $\Kh(K)$ of topological/geometric significance: 
   In  \S\ref{sec:examples}, our constructions are guided by studying symmetries of $K$ that fail to extend over the surfaces it bounds.   (See \cite{ls:invertible,bdms} for deeper investigation of the equivariant perspective.) In \S\ref{sec:plam}, we offer a second perspective using braids and complex curves, as discussed below. These strategies have since been successfully applied to other problems, e.g., distinguishing Seifert surfaces  \cite{hkmps:seifert} and satellite surfaces in \cite{guth-hayden-kang-park}.

\smallskip

\textbf{Connections with braids and Plamenevskaya's invariant} \ With an eye towards a more systematic and geometric approach to these cobordism maps, we  develop computational techniques from a braid-theoretic perspective. A natural starting point is Plamenevskaya's invariant of \emph{transverse links} \cite{plamenevskaya:transverse-Kh}, i.e., oriented links that are positively transverse to the planes of the standard contact structure on $S^3$. This detour through contact geometry is motivated by the fact that many of our surfaces arise as the transverse intersection of a smooth complex curve in $\mathbb{C}^2$ with the unit 4-ball; the boundary is then a transverse link in $S^3$ (c.f., \cite{bo:qp,hayden:stein}). This includes the disks in Figure~\ref{fig:main-disks} bounded by $J=17nh_{74}$, as well as those bounded by the knots $m(9_{46})$ and $15n_{103488}$ in Figure~\ref{fig:further} and by $10_{148}$ in Figure~\ref{fig:10-148}. This connection between surfaces in $B^4$ and complex curves is expressed  using Rudolph's framework of \emph{braided surfaces} \cite{rudolph:braided-surface,rudolph:qp-alg}, which we review in \S\ref{sec:plam}; see Figure~\ref{fig:braided-surfaces} for an example.

We show that Plamenevskaya's invariant behaves naturally under the maps induced by compact pieces of complex curves. This applies more broadly to \emph{ascending cobordisms with positive critical points}, a class of surfaces that generalize complex curves; see \S\ref{sec:plam}.

\begin{figure}\center
\includegraphics[width=\linewidth]{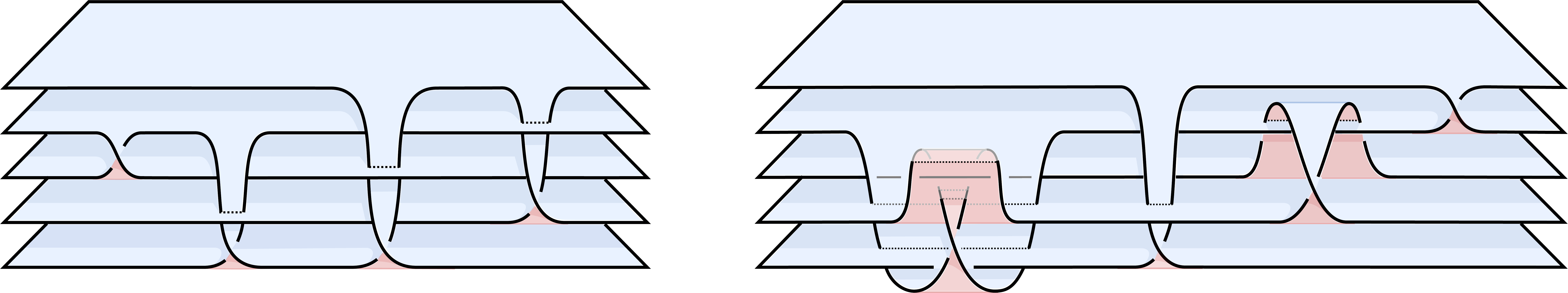}
\caption{Positively braided surfaces representing the disks $D$ (left) and $D'$ (right).}
\label{fig:braided-surfaces}
\end{figure}


\begin{thm}\label{thm:plam}
Suppose $\Sigma \subset S^3 \times [0,1]$ is an ascending cobordism with positive critical points between transverse links $L_0$ and $L_1$ in $(S^3,\xi_\st)$, viewed as a link cobordism from $L_1$ to $L_0$. The induced cobordism map $\Kh(\Sigma): \Kh(L_1) \to \Kh(L_0)$ maps $\psi(L_1)$ to $\pm \psi(L_0)$, where $\psi(L_i)$ denotes Plamenevskaya's invariant in $\Kh(L_i)$.
\end{thm}

While Theorem~\ref{thm:plam} can be useful (e.g., implying certain cobordism maps are nonzero or obstructing a surface from being isotopic to a complex curve), it also shows that the behavior of Plamenevskaya's invariant is too uniform to distinguish such surfaces. 

Instead, we pursue a modified construction to  produce Khovanov homology classes that are tuned to individual surfaces.  Plamenevskaya's construction is based on expressing a link as the closure of a braid, but it essentially depends only on the braid's equivalence class in the braid group (see \S\ref{subsec:plam}). Our modified construction depends on a specific presentation of the braid (known as a \emph{band factorization}) that also encodes a braided surface bounded by the link. To illustrate this approach, we give an alternate proof that the core disks $D$ and $D'$ from Figure~\ref{fig:main-disks} are distinct, as well as examples of distinct complex curves in $B^4$ bounded by the knot $10_{148}$.

\smallskip

\textbf{Connections with knot Floer homology} \  
The TQFT structure of a perturbed version of knot Floer homology has also been shown to distinguish exotically knotted surfaces of positive genus in the 4-ball \cite{jmz:exotic}. (Roughly speaking, this perturbed theory replaces $\hfk(K)$ with a perturbed group $\hfk(K) \otimes \mathbb{F}_2[\zz^{2g}]$.) The perturbed  cobordism maps are sensitive enough to detect the subtle effect of a twisted version of Fintushel-Stern's rim surgery construction \cite{fs:surfaces}. We also note that an analog of Theorem~\ref{thm:plam} was established for the transverse invariant in knot Floer homology in \cite{jmz:exotic}, and it plays a key role in their detection of exotically knotted surfaces. 

On the other hand, the obstruction from \cite{jmz:exotic} cannot distinguish between exotic pairs of complex curves. These differences spark several questions about the TQFT structures of knot Floer homology and Khovanov homology, including: (1) Can the cobordism maps in Khovanov homology distinguish surfaces related by rim surgery? (2) Can the cobordism maps in the standard (unperturbed) knot Floer theory distinguish exotically knotted surfaces in the 4-ball? (3) Can any version of knot Floer homology distinguish exotically knotted pairs of disks or complex curves in the 4-ball? 


\smallskip

\emph{Acknowledgements.}  The authors thank Alan Du and Gage Martin for helpful input and Selman Akbulut for identifying connections with the literature and, in particular, Remark \ref{remark:akbulut}. KH is supported by NSF grants DMS-1803584 and DMS-2114837.

\section{Preliminaries}


We begin with some background on link cobordisms and Khovanov homology, focusing on the Khovanov chain complex and the chain maps induced by a given movie of the link cobordism. Throughout the paper, we work with integral coefficients.

\subsection{Link cobordisms.}

A \textit{link cobordism} is a smooth, compact, oriented, properly embedded surface $\Sigma \subset \rr^3 \times [0,1]$ whose boundary is a pair of oriented links $L_0 \sqcup L_1 = \Sigma \cap (\rr^3 \times \{0,1\})$. We often write the link cobordism as a function $\Sigma : L_0 \to L_1$. 

To study a given link cobordism, we represent it as a \textit{movie}, that is, a finite sequence of link diagrams $D_0 = D_{t_0}, D_{t_1}, \dots, D_{t_n} = D_1$ having two properties: the boundary links $L_0$ and $L_1$ of the link cobordism are represented by the first and last diagrams $D_0$ and $D_1$ in the sequence; successive pairs of diagrams are related by a planar isotopy, Reidemeister move, or Morse move. From an arbitrary link cobordism, one can write down an associated movie (c.f., \cite[\S 3]{jacobsson}), however, in practice, we often choose a movie and study the link cobordism defined by the trace of the given moves. 

\subsection{Khovanov homology.}
Given a diagram $D$ of an oriented link $L$ and an enumeration of its crossings, we associate a chain complex $(\CKh(D), d)$ called the \textit{Khovanov chain complex}. We describe it here, attempting to avoid any cumbersome algebra. 

A \textit{smoothing} $\sigma$ is a planar $1$-manifold obtained by replacing each crossing $\crossing$ in $D$ with either a $0$-smoothing $\zsmooth$ or a $1$-smoothing $\osmooth$ \!. Using the enumeration of the crossings, $\sigma$ can be represented as a binary sequence $(\sigma_1, \dots, \sigma_n)$, where $\sigma_i \in \{0,1\}$ indicates that the $i$-th smoothing is $\sigma_i$-smoothed. We say a loop (i.e., connected component) in $\sigma$ is $0$-\textit{tracing} (or $1$-\textit{tracing}) if it intersects a $0$-smoothing (or $1$-smoothing). A \textit{labeled smoothing} $\alpha_\sigma$ is a labeling of the loops of the smoothing $\sigma$ with a $1$ or $x$. The chain group $\CKh(D)$ is generated over $\zz$ by the labeled smoothings of $D$.

The chain complex is bigraded $\CKh^{h,q}(D)$ by \textit{homological} grading $h$ and \textit{quantum} grading $q$.  Let $n_+$ and $n_-$ record the number of positive and negative crossings in $D$; let $|\sigma|$ record the number of $1$-smoothings in $\sigma$; let $v_+$ and $v_-$ record the number of $1$-labels and $x$-labels in $\alpha_\sigma$. Then $h$ and $q$ are defined on $\alpha_\sigma$ by
\begin{align*}
	h(\alpha_\sigma) &= |\sigma| - n_- \\
	q(\alpha_\sigma) &= v_+(\alpha_\sigma) - v_-(\alpha_\sigma) + h(\alpha_\sigma) + n_+ - n_-
\end{align*}

For a labeled smoothing $\alpha_\sigma$, the differential $d(\alpha_\sigma)$ will be a $\zz$-linear combination of labeled smoothings obtained as follows. 
First, consider the binary representation of the smoothing $\sigma = (\sigma_1, \dots, \sigma_n)$, and for each $i$ such that $\sigma_i = 0$, let $\sigma^i$ be the smoothing obtained by setting $\sigma_i = 1$. Note that $\sigma$ and $\sigma^i$ cobound a surface that is a product away from the $i$-th crossing, where it is a single Morse saddle. A	 labeled smoothing $\alpha_{\sigma^i}$ is obtained by applying the corresponding Morse induced chain map from Table \ref{table_morse} to $\alpha_\sigma$.  Let $\xi^i = \sum_{j<i} \sigma_j$. We then define the differential by the following formula: 
	$$d(\alpha_\sigma) = \sum_{\{i \, | \, \sigma_i = 0\}} (-1)^{\xi^i} \alpha_{\sigma^i}$$
The homology $\Kh(D)$ of the chain complex $(\CKh(D), d)$ is called the \textit{Khovanov homology}. Different diagrams for the same link yield isomorphic Khovanov homology groups. In later sections, we will write $\Kh(L)$ in place of $\Kh(D)$; the diagram $D$ in use will be clear from context. We recycle the notation of a cycle $\cycle \in \CKh(D)$ for the homology class it represents $\class \in \Kh(D)$, with membership being clear from context. In this work, we mainly consider Khovanov homology classes represented by a single labeled smoothing. We use the following to check whether a labeled smoothing is a cycle; it follows quickly from the definition of the differential (c.f., \cite[Prop. 3.2]{elliott}).

\begin{prop}\label{prop:cycle}
	A labeled smoothing $\alpha_\sigma$ is a cycle if and only if every $0$-smoothing in $\sigma$, when changed to a $1$-smoothing, merges two $x$-labeled loops.
\end{prop}

For convenience, we will occasionally record the location of the $0$-smoothings in a given labeled smoothing $\alpha_\sigma$ by decorating each $0$-smoothing with a light gray arc that connects the relevant strands in the smoothing (e.g., see Figure \ref{fig:phi}). To check if $\alpha_\sigma$ is a cycle, it suffices to check if each arc connects a pair of distinct, $x$-labeled loops.

\subsection{Induced maps on Khovanov homology.} \label{subsec:induced_maps}

Given a pair of diagrams $D_0$ and $D_1$ representing the boundary links $L_0$ and $L_1$ of an oriented link cobordism $\Sigma : L_0 \to L_1$, we may associate a bigraded chain map
	$$\CKh(\Sigma) : \CKh^{h,q}(D_0) \to \CKh^{h,q+\chi(\Sigma)}(D_1)$$
with induced homomorphism $\Kh(\Sigma) : \Kh(D_0) \to \Kh(D_1)$. This paper hinges on the following invariance theorem proven by Jacobsson (c.f., \cite{khovanov:invariant,barnatan}).

\begin{thm}[\cite{jacobsson}] \label{thm:jacobsson}
	The homomorphism $\Kh(\Sigma) : \Kh(D_0) \to \Kh(D_1)$ is invariant up to multiplication by $\pm1$ under smooth isotopy of $\Sigma$ fixing $\partial \Sigma$ setwise.
\end{thm} 

\begin{rem}
	
In this work, we study smooth, compact, oriented, properly embedded surfaces $\Sigma \subset B^4$ with boundary $L = \partial \Sigma$. To tailor the Khovanov invariant to these surfaces, we note that an analogous version of Theorem \ref{thm:jacobsson} was proven in \cite{morrison-walker-wedrich} for surfaces in $S^3 \times [0,1]$. Any isotopy of $\Sigma$ through $B^4$ induces an isotopy of surfaces in $S^3 \times [0,1]$ (e.g., by removing an open ball  in the complement of the support of the isotopy), so invariance extends naturally to our setting. However, note that the identification of $\partial B^4$ with $S^3 \times \{1\}$ or $S^3 \times \{0\}$ will produce different link cobordisms, with the former yielding a link cobordism $\Sigma : \emptyset \to L$ and the latter $\Sigma : L \to \emptyset$. Given a movie for one of these link cobordisms, we may produce a movie of the other by reversing the order of the diagrams.
\end{rem}

We now discuss the definition of the chain map $\CKh(\Sigma)$, attempting to avoid any cumbersome algebra. We follow the process outlined in \cite{barnatan}. The idea is to first define chain maps induced by the three diagrammatic relations used in a movie of $\Sigma$. Then, given a movie $D_0 = D_{t_0}, D_{t_1}, \dots, D_{t_n} = D_1$ of $\Sigma$, we produce a collection of chain maps induced by successive pairs of diagrams $\CKh(D_{t_i}) \to \CKh(D_{t_{i+1}})$. The desired chain map $\CKh(\Sigma)$ is the successive composition of these chain maps. It suffices to give explicit definitions for the chain maps induced by each of the diagrammatic relations: planar isotopies, Morse moves, and Reidemeister moves.

\textbf{Isotopy induced chain maps} \ The chain map induced by an isotopy of diagrams is defined on a labeled smoothing $\alpha_\sigma$ by applying the isotopy to the underlying smoothing $\sigma$ and preserving the labeling from $\alpha_\sigma$ of the components in $\sigma$ throughout this isotopy.

\textbf{Ornaments} \ We pause to develop a convenient shorthand from \cite{barnatan}. 
The Morse and Reidemeister moves only change a diagram locally within some tangle. As a result, for a labeled smoothing $\alpha_\sigma$, it suffices to define the induced chain maps on the portion of $\alpha_\sigma$ within this tangle, while leaving the rest of the labeled smoothing unchanged. In order to properly define the chain map, we must account for all possible smoothings of the tangle, as well as all possible labels for each smoothing. As a result, it is  convenient to have a shorthand that simplifies the amount of information necessary to express these maps. The idea is to reduce the definition to the level of smoothings by defining a set of local \textit{ornaments} that can be placed on a smoothing, each of which corresponds to a predetermined chain map on the portion of the smoothing it adorns. A chain map can then be defined on all possible labelings of a smoothing $\sigma$ by simply decorating $\sigma$ with these ornaments: to any given labeled smoothing $\alpha_\sigma$, apply each of the predetermined chain maps corresponding to the ornaments decorating $\sigma$.

The ornaments we need correspond, perhaps not surprisingly, to the three Morse moves: births, deaths, and saddles. A birth will locally add a crossingless unknot to an empty tangle; we decorate a smoothing with the ornament $\kcup$ consisting of a crossingless unknot with $4$ external antennae to indicate this addition. Similarly, a death removes a crossingless unknot, in which case we decorate the smoothing with the ornament $\kcap$ consisting of $4$ internal antennae adorning the component being removed. A saddle acts on a tangle with two unknotted arcs $\osmooth$ by either merging or splitting the component(s) to which the arcs belong; in either case, the result is a tangle $\zsmooth$. We decorate the smoothing with the ornament $\ksmooth$ consisting of a thin line that perpendicularly intersects the two components being merged or split. The chain maps induced by these ornaments are defined locally in Table \ref{table_morse}.

\begin{table}[!ht]\center
		\renewcommand{\arraystretch}{2.25}
		\begin{tabular}{|c|c|c|l|}
			\hline
			\text{Morse Move} & \text{Ornament} & \text{Chain map} &  \multicolumn{1}{c|}{\text{Definition of chain map}} \\
			\hline
			birth & \kcup & $\iota$ & $\begin{array}{l} \hspace{1.25em} 1 \hspace{1.3em} \mapsto \kocirc \vspace{.425em} \end{array}$ \\
			\hline
			death & \kcap & $\varepsilon$ & $\begin{array}{l} \kocirc \mapsto \hspace{1.25em} 0 \\ \kxcirc \mapsto \hspace{1.25em} 1 \vspace{.425em} \end{array}$ \\
			\hline
			\multirow{2}{80pt}{\hspace{2.15em}\vspace{-5em}saddle} & $\begin{array}{l} \vspace{-6em}\ksmooth \end{array}$ & $m$ & $\begin{array}{l} \kcobdoo \mapsto \kcobio \\ \kcobdox \mapsto \kcobix \\ \kcobdxo \mapsto \kcobix \\ \kcobdxx \mapsto \hspace{1.25em} 0 \vspace{.425em} \end{array}$ \\\cline{3-4} & & $\Delta$ & $\begin{array}{l} \kcobdo \mapsto \kcobiox + \kcobixo \\ \kcobdx \mapsto \kcobixx \vspace{.425em} \end{array}$ \\
			\hline
		\end{tabular}
		\renewcommand{\arraystretch}{1}
	\caption{The chain maps induced by Morse moves.\vspace{-5pt}}
	\label{table_morse}
\end{table}

One additional decoration $\kdot$ will be needed, consisting of a dot on any component of the smoothing. This decoration indicates the application of two saddles (one splitting and then one re-merging) on the decorated component. Using Table \ref{table_morse}, one can verify that the map induced by this local cobordism kills an $x$-labeled arc, but sends a $1$-labeled arc to twice an $x$-labeled arc.

\textbf{Morse induced chain maps} \ The chain map induced by a Morse move is defined on a labeled smoothing $\alpha_\sigma$ by decorating the smoothing $\sigma$ with the ornament corresponding to the given Morse move.

\textbf{Reidemeister induced chain maps} \ The chain map induced by a Reidemeister move is defined on a labeled smoothing $\alpha_\sigma$ by decorating the smoothing $\sigma$ with the ornaments given in Tables \ref{table_r1}-\hyperlink{table_r3}{5} in Appendix \ref{subsec:chain_maps}. As a given decoration can consist of multiple ornaments, there is a natural question of the order in which the corresponding chain maps should be applied; this will either be irrelevant (i.e., the moves and their induced maps commute) or clear from context (e.g., a dotted arc on a birth requires the birth to occur before the map induced by the dotted arc can be applied).

In this paper, we only use use complexity-reducing Reidemeister I and II moves (those that remove crossings). We list these chain maps here, in Table \ref{table_reidemeister_redux}.

\begin{table}[!ht]\center
	\renewcommand{\arraystretch}{2.25}
	\begin{tabular}{|c|c|c|}
		\hline
		Reidemeister move & Smoothing & Induced chain map \\
		\hline
		\multirow{2}{80pt}{\begin{minipage}{9em}
				$\kropos \to \kroarc$
			\end{minipage}} & \krosa & \kroma $\begin{array}{c}\hspace{-1em}\vspace{.425em}\end{array}$ \\\cline{2-3} & \krosb & $\begin{array}{c} 0 \vspace{.425em} \end{array}$ \\
		\hline
		\multirow{2}{80pt}{\begin{minipage}{9em}
				$\kroneg \to \kroarc$
		\end{minipage}} & \krosa & \raisebox{.1em}{$\begin{array}{c} \frac12 \Bigg( \kromc \Bigg) \vspace{.25em} \end{array}$} \\\cline{2-3} & \krosb & $\begin{array}{c} 0 \vspace{.425em} \end{array}$ \\
		\hline
		\multirow{2}{80pt}{\begin{minipage}{9em}
				$\krtcr \to \krtcrl$
		\end{minipage}} & \krtsa & $-$ \!\! \krtma $\begin{array}{c}\hspace{.25em}\vspace{.425em}\end{array}$ \\\cline{2-3} & \krtsb & \krtcrl $\begin{array}{c}\hspace{-1em}\vspace{.425em}\end{array}$ \\\cline{2-3} & \krtsc & $\begin{array}{c} 0 \vspace{.425em} \end{array}$ \\\cline{2-3} & \krtsd & $\begin{array}{c} 0 \vspace{.425em} \end{array}$ \\
		\hline
	\end{tabular}
	\renewcommand{\arraystretch}{1}
	
	\caption{The relevant chain maps induced by Reidemeister I and II moves.\vspace{-5pt}}
	\label{table_reidemeister_redux}
\end{table}

\begin{rem}
	Note the $\frac12$ in the definition of the Reidemeister I induced chain map in Table \ref{table_reidemeister_redux} does not conflict with $\zz$ as our coefficient group: the dotted arc will always produce an even coefficient, so overall, the map will maintain an integral coefficient.
\end{rem}

\subsection{Local knotting.} The cobordism-induced maps are invariant under boundary-preserving isotopy as well as another operation: a link cobordism is \textit{locally knotted} if it can be written as $\Sigma \# S$ for a surface $\Sigma$ and a knotted $2$-sphere $S \subset \rr^3 \times [0,1]$. Locally knotting a surface will generally change the boundary-preserving isotopy class of the surface. The following theorem guarantees that any detection obtained by the cobordism-induced maps on Khovanov homology is not due to this simple operation.

\begin{thm}
	The cobordism-induced maps on Khovanov homology are invariant under local knotting: given a link cobordism $\Sigma : L_0 \to L_1$ and a knotted $2$-sphere $S$, the induced maps $\Kh(\Sigma)$ and $\Kh(\Sigma \# S)$  agree up to multiplication by $\pm1$.
\end{thm}

\begin{proof}
In the case where $L_0 = \emptyset$, the induced map $\Kh(L_0) \to \Kh(L_1)$ is determined by the relative Khovanov-Jacobsson class of the surface. By \cite[Theorem~4.2]{sundberg-swann}, relative Khovanov-Jacobsson classes are invariant under local knotting. 
This argument can be adapted to the case where $L_0 \neq \emptyset$. Let $B$ be a $4$-ball intersecting $\Sigma \# S$ along the disk $S \setminus \mathring{D}^2$ bounded by an unknot $U$ in $\partial B \cong S^3$. We may perform a boundary-preserving isotopy of $\Sigma \# S$ that drags $B$ near $L_0$. It then suffices to show that locally knotting the product cobordism $C: L_0 \to L_0$ induces the identity map. We can isolate $B$ so that $C \# S$ decomposes into a link cobordism $C \sqcup (S \setminus \mathring{D}^2): L_0 \to L_0 \sqcup U$ followed by a saddle merging $L_0$ and $U$. By \cite{sundberg-swann}, the map induced by $S \setminus \mathring{D}^2$ is identical to the map induced by the link cobordism induced by a standard $D^2$ in $B$. Moreover, the map on Khovanov homology induced by a split cobordism will split as the tensor product of the individual cobordism-induced maps, so $C \sqcup (S \setminus \mathring{D}^2)$ induces the same map as $C \sqcup D^2$. Stacking the saddle on the latter cobordism yields a surface isotopic to $C$ rel boundary, so by Theorem~\ref{thm:jacobsson} they induce the same map, as desired.
\end{proof}


\section{Distinguishing cobordism maps}\label{sec:examples}


In this section, we obstruct the smooth, boundary-preserving isotopy of pairs of surfaces $\Sigma,\Sigma'$ bounded by a common knot $K$ by viewing them as cobordisms $K \to \emptyset$ and distinguishing their associated induced maps $\Kh(\Sigma),\Kh(\Sigma'):\Kh(K) \to \zz$, which are invariants of smooth boundary-preserving isotopy by Theorem~\ref{thm:jacobsson}. In particular, we provide a class $\class \in \Kh(K)$ which is mapped to $1$ under $\Kh(\Sigma)$ and $0$ under $\Kh(\Sigma')$. 

\begin{rem}\label{rem:finding-phi}
At present, producing such cycles $\cycle$ is more art than science. 
We typically began with the orientation-induced smoothing where $0$-tracing loops are $x$-labeled and all other loops are $1$-labeled.\footnote{In certain cases, this labeled smoothing corresponds to Plamenevskaya's invariant.} This labeled state is always a cycle lying in homological grading $h=0$. However, it may not have the desired quantum grading; a surface $\Sigma$ induces a $(0, \chi(\Sigma))$-graded map, so a cycle must lie in $\CKh^{0,-\chi(\Sigma)}(K)$ in order for it to be mapped to the $(0,0)$-supported chain group $\CKh(\emptyset) = \zz$. While the homological grading (and the underlying diagram) determines the overall balance of 0- and 1-resolutions, the quantum grading can be adjusted by varying the specific choice of crossing resolutions (which may change the number of loops in the smoothing) and the labeling of loops. We made these adjustments keeping in mind that the result should be a cycle and should be killed by the map induced by one band move but not the other. In our core cases, the slice disks are related by a symmetry of the knot; making asymmetric adjustments to the orientation-induced smoothing helped produce the desired cycle.
\end{rem}

\subsection{A core example.}\label{subsec:core} We now distinguish the cobordism maps induced by the disks $D$ and $D'$ from Figure~\ref{fig:main-disks}. The surfaces in Section \ref{subsec:ribbon-genus} and Theorem \ref{thm:main} are all extensions of this initial example, as are the Khovanov-theoretic computations that distinguish them. 

\begin{thm}\label{thm:diff}
	The disks $D$ and $D'$ induce distinct maps on Khovanov homology, hence are not smoothly isotopic rel boundary.
\end{thm}

\begin{proof}
	The left side of Figure~\ref{fig:phi} depicts the knot $J$ (decorated with bands $b$ and $b'$ corresponding to the disks $D$ and $D'$), while the right side of the figure depicts a distinguished chain element $\cycle \in \CKh(J)$. By Proposition~\ref{prop:cycle}, the chain $\cycle$ is a cycle; in particular, all arcs corresponding to 0-smoothings join two distinct $x$-labeled loops.
	
	\begin{figure}\center
		\def\svgwidth{.65\linewidth}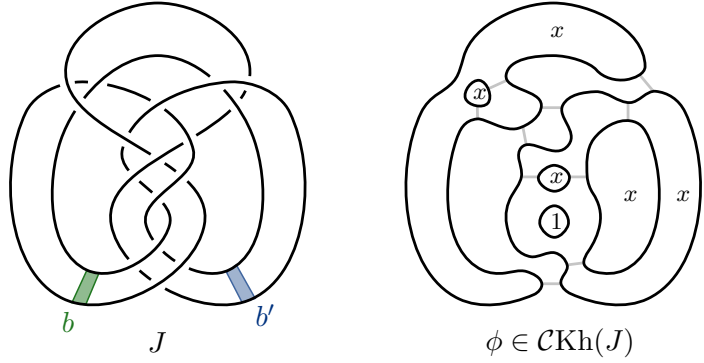
		\caption{The knot $J$ with band moves $b$ and $b'$ describing the pair of slices $D$ and $D'$, together with a cycle $\cycle$ distinguishing their induced maps on Khovanov homology.}\label{fig:phi}
	\end{figure}

	We claim that $\cycle$ vanishes under the map induced by the cobordism $D': J \to \emptyset$. This cobordism begins with a saddle move along the band $b'$. The associated cobordism map merges two distinct $x$-labeled loops of $\cycle$, hence maps $\cycle$ to $0$. On the other hand, we claim that $\cycle$ is mapped to $1 \in \zz = \CKh(\emptyset)$  under the map induced by the cobordism $D: J \to \emptyset$. This calculation is carried out in Figure~\ref{fig:cob-map}.
\end{proof}

\subsection{Further examples.}\label{subsec:further} A similar technique can be applied to other pairs of slices. We give two such examples in Figure~\ref{fig:further}, which depicts pairs of slices of the knots $m(9_{46})$ and $15n_{103488}$. In each case, we provide a knot diagram decorated with a pair of bands describing the slices, as well as a cycle $\cycle$ in the chain complex associated to the diagram. As before, one slice will kill $\cycle \mapsto 0$ and the other sends $\cycle \mapsto 1$. 

\vspace{12pt}
\begin{figure}[h]\center
	\def\svgwidth{\linewidth}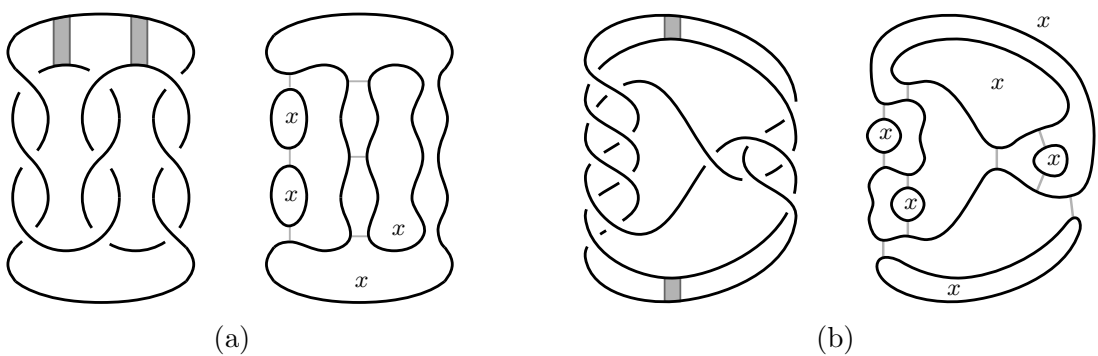
	\vspace{-4pt}
	\caption{The knots (a) $m(9_{46})$ and (b) $15n_{103488}$ with band moves describing slices for each knot, distinguished by the behavior of their induced maps on the given cycle.}\label{fig:further}
\end{figure}

Note that these slices of $m(9_{46})$ and $15n_{103488}$ can be distinguished by their \emph{peripheral maps} (borrowing terminology from  \cite[Definition~3.9]{juhaszzemke}), i.e., the map on fundamental groups induced by including the knot complement into the slice disk complement. It follows that these pairs of slices are not even \emph{topologically} isotopic rel boundary.
%

\begin{figure}
	\hspace{-.25in} \includegraphics[width=1.1\linewidth]{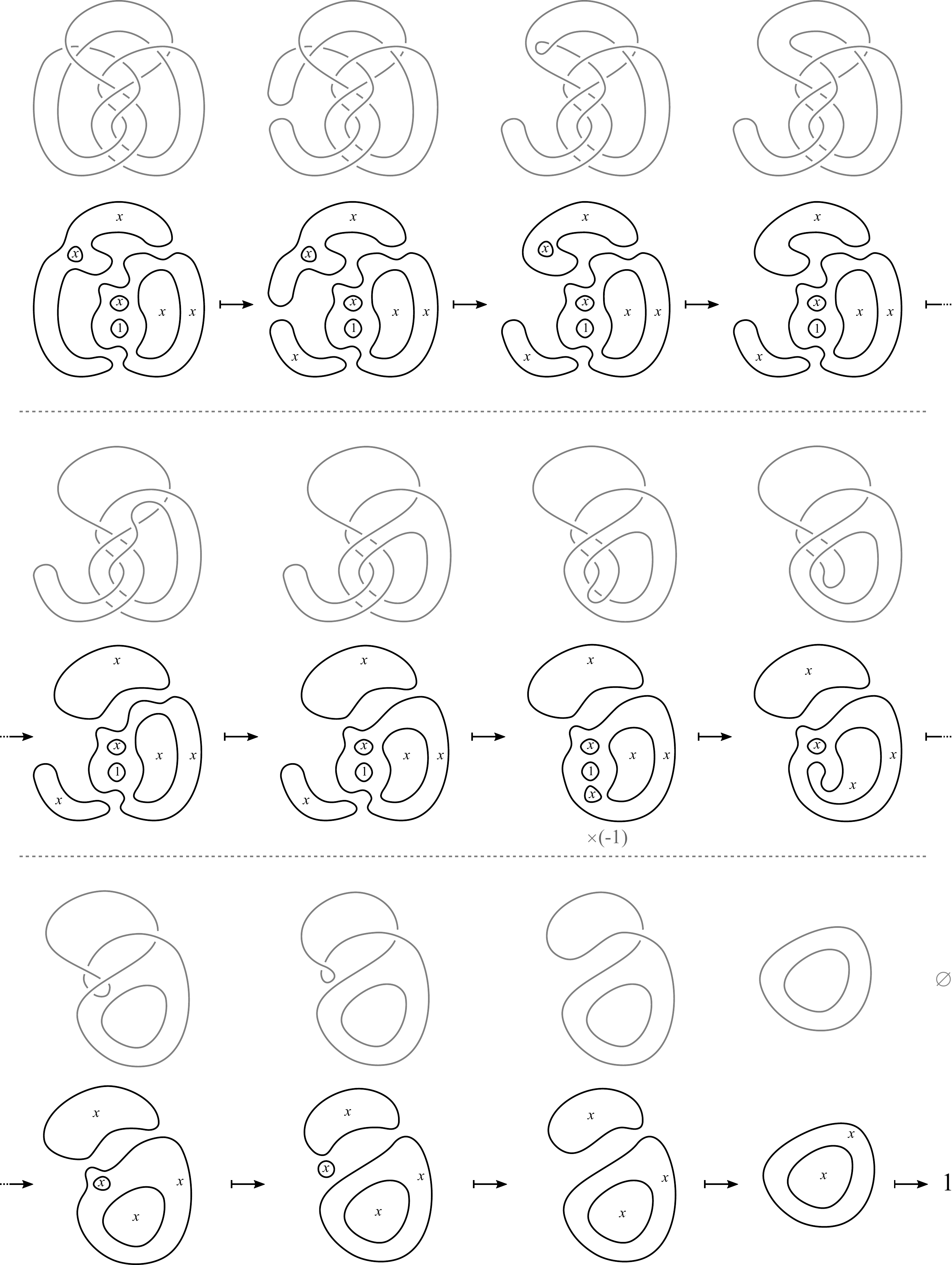}
	\caption{A movie description of the slice disk $D$ and the behavior of the distinguished cycle $\cycle \in \CKh(J)$ under the cobordism map induced by this slice.}\label{fig:cob-map}
\end{figure}

\subsection{Ribbon concordance and higher-genus examples.} \label{subsec:ribbon-genus} The obstruction described above is robust in the sense that it persists when a surface is enlarged by a ribbon concordance (i.e., a link concordance that has no local maxima) or, in many cases, by adding positively twisted bands to increase the genus of the surfaces.

\begin{thm}[Levine-Zemke \cite{levine-zemke}]
If $C$ is a ribbon concordance from $L_1$ to $L_2$, the induced map $\Kh(C): \Kh(L_1) \to \Kh(L_2)$ is injective, 
with left inverse induced by the reverse of $C$, viewed as a cobordism from $L_2$ to $L_1$.
\end{thm}

\begin{cor}\label{cor:ribbon}
Let $\Sigma$ and $\Sigma'$ be cobordisms from $L_0$ to $L_1$ and let $C$ be a ribbon concordance from $L_1$ to $L_2$. If $\Sigma$ and $\Sigma'$ induce distinct maps $\Kh(L_0) \to \Kh(L_1)$, then the cobordisms $C \circ \Sigma$ and $C \circ \Sigma'$ induce distinct maps $\Kh(L_0) \to \Kh(L_2)$.

Similarly, if the reverses of $\Sigma$ and $\Sigma'$ induce distinct maps $\Kh(L_1) \to \Kh(L_0)$, then the reverses of $C \circ \Sigma$ and $C \circ \Sigma'$ induce distinct maps $\Kh(L_2) \to \Kh(L_0)$.
\end{cor}


\begin{proof}
If $\Sigma$ and $\Sigma'$ induce distinct maps on Khovanov homology when viewed as cobordisms $L_0 \to L_1$, there must be an element $\alpha \in \Kh(L_0)$ such that $\Kh(\Sigma)(\alpha) \neq \Kh(\Sigma')(\alpha)$. Since $C$ induces an injective map $\Kh(L_1) \to \Kh(L_2)$, we have
\begin{equation*}
\Kh(C \circ \Sigma)(\alpha)-\Kh(C \circ \Sigma')(\alpha) =\Kh(C)\left(\Kh(\Sigma)(\alpha)-\Kh(\Sigma')(\alpha)\right) \neq 0.
\end{equation*}
An analogous argument applies to the reversed cobordisms, appealing instead to the surjectivity of the map $\Kh(L_2)\to \Kh(L_1)$ induced by the reverse of $C$.
\end{proof}

\begin{rem}
A similar (independently established) technique is used in \cite{sundberg-swann} for finding prime knots with an arbitrarily large number of distinct (but non-exotic) slices. Moreover, a similar technique appears in \cite{juhaszzemke} for an invariant from \cite{juhaszmarengon} in knot Floer homology.
\end{rem}

\begin{ex}[Asymmetric slice knots]\label{ex:asym} 

For any $m \in \zz$, there is a ribbon concordance $C$ from $J$ to the knot $J_m$ depicted on the left side of Figure~\ref{fig:asym}; in reverse, we obtain $J$ from $J_m$ by performing the gray band move, which splits off an unknot that is capped with a disk.  By Corollary~\ref{cor:ribbon}, the slice disks $D_m$ and $D_m'$ obtained by gluing $C$ to $D$ and $D'$, respectively, induce distinct maps on Khovanov homology. In fact, for $m \geq 0$, it is straightforward to identify a class in $\Kh(J_m)$ that distinguishes these maps (see Figure~\ref{fig:asym}), whereby $D_m$ and $D_m'$ are not smoothly isotopic rel boundary.

Unlike the examples in \S\ref{subsec:core}-\ref{subsec:further}, which involve slice knots with nontrivial symmetries, the knots $J_m$ are \emph{asymmetric}. That is, every self-diffeomorphism of the pair $(S^3,J_m)$ is isotopic (through diffeomorphisms of the pair) to the identity. This is proven in the appendix \S\ref{subsec:asym} for $m \gg 0$, but similar arguments to the ones given there also establish the claim for all $m$. \hfill $\diamond$
\end{ex}

\begin{figure}[b]\center
	\def\svgwidth{.925\linewidth}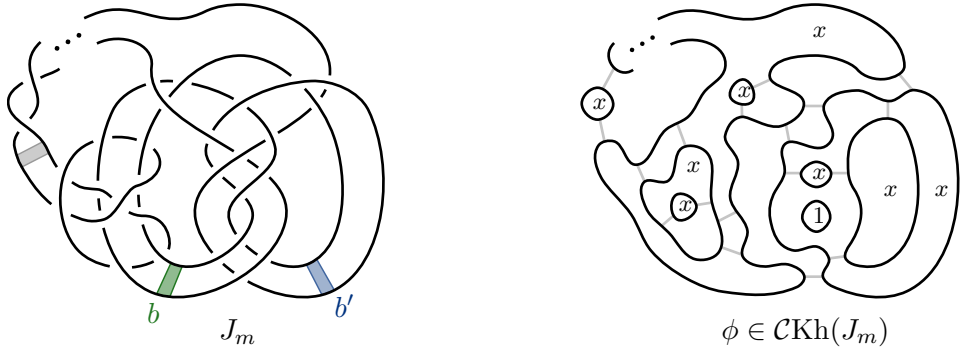
	\caption{An asymmetric version $J_m$ of the knot $J$, having slices $D_m$ and $D_m'$ distinguished by the behavior of their induced maps on the given cycle.}\label{fig:asym}
\end{figure}

\begin{ex}[Higher-genus surfaces] \label{ex:higher-genus}
Fix integers $m,n \geq 0$ and let $J_{m,n}$ be the knot shown on the left side of Figure~\ref{fig:higher-genus-clasp}; there are a total of $m$ full left-handed twists on the left side of the knot and $n$ full right-handed twists on the right side of the knot. 

Performing $2n$ saddle moves (along the gray bands shown on the left side of Figure~\ref{fig:higher-genus-clasp}) yields a cobordism of genus $n$ from $J_{m,n}$ back to the knot $J_m$ from Example~\ref{ex:asym}. It is straightforward to check that the map induced by this cobordism sends the cycle $\theta \in \CKh(J_{m,n})$ shown on the right side of Figure~\ref{fig:higher-genus-clasp} to the cycle $\cycle \in \CKh(J_m)$. Gluing this cobordism to the disks $D_m$ and $D_m'$ bounded by $J_m$ yields a pair of slice surfaces of genus $n$ for $J_{m,n}$, which we denote by $\Sigma_{m,n}$ and $\Sigma'_{m,n}$, respectively. By composing the cobordism maps, we see that $\Kh(\Sigma_{m,n})(\theta) = 1$ and $\Kh(\Sigma'_{m,n})(\theta) = 0$. It follows that $\Sigma_{m,n}$ and $\Sigma'_{m,n}$ are not smoothly isotopic rel boundary. \hfill $\diamond$

\end{ex}
\medskip

\begin{ex}[Boundary-sums] \label{ex:connected-sums}
Our calculations all extend to boundary-sums. In the above calculations, we give pairs of surfaces $\Sigma$ and $\Sigma'$ with boundary $K$ distinguished by a class $\phi \in \Kh(K)$. For $n > 0$, we encourage the reader to produce a cycle $\#_n\phi \in \Kh(\#_nK)$ that distinguishes the $2^n$ surfaces bounding $\#_nK$, obtained by boundary-summing different collections of $\Sigma$ and $\Sigma'$. \hfill $\diamond$
\end{ex}

\begin{figure}[t]\center
\def\svgwidth{\linewidth}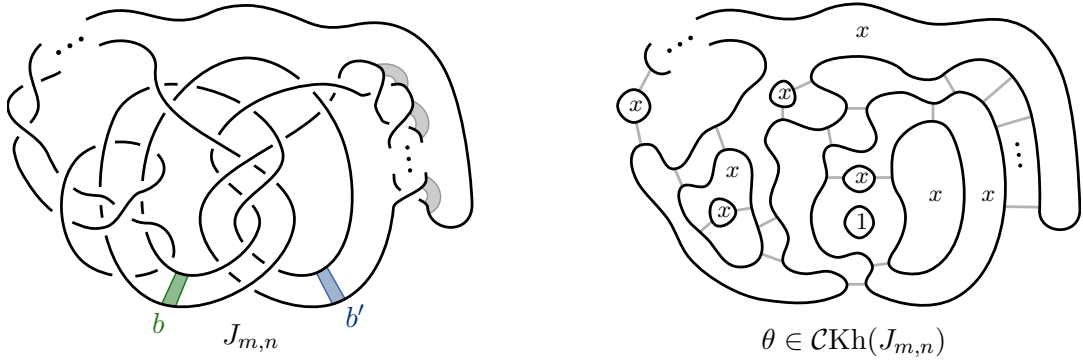
\caption{A higher-genus version $J_{m,n}$ of $J_m$, bounding surfaces $\Sigma_{m,n}$ and $\Sigma_{m,n}'$ distinguished by the behavior of their induced maps on the given cycle.\vspace{-10pt}}\label{fig:higher-genus-clasp}
\end{figure}

\section{Exotically knotted surfaces}


In this section, we prove Theorem~\ref{thm:main}. Our surfaces will be drawn from the examples in \S\ref{sec:examples}, especially the core examples $D$ and $D'$ from Figure~\ref{fig:main-disks}, hence the Khovanov-theoretic obstructions are already in place. Thus our discussion will focus on two remaining problems: (1) showing that the surfaces in question are topologically isotopic rel boundary, and (2) distinguishing the surfaces up to smooth isotopies of $B^4$ that do not fix the boundary. For the first task, we will rely on the following result of Conway and Powell.

\begin{thm}[\cite{conway-powell}]\label{thm:conway-powell}
Any smooth, properly embedded disks in $B^4$ with the same boundary and whose complements have $\pi_1 \cong \zz$ are topologically isotopic rel boundary.
\end{thm}

\begin{prop}\label{prop:top}
The  slice disks $D$ and $D'$ are topologically isotopic rel boundary.
\end{prop}

\begin{proof}
By construction, the disks $D$ and $D'$ have the same boundary. By Theorem~\ref{thm:conway-powell}, it then suffices to show that the disk exteriors have $\pi_1 \cong \zz$.  

A handle diagram for the first disk exterior $B^4 \setminus \mathring{N}(D)$ is shown on the left side of Figure~\ref{fig:pi1}, obtained using the recipe from \cite[\S6.2]{GompfStipsicz4}. To simplify our calculation, we recall that $\pi_1$ is not changed under homotopy of the attaching curves for 2-handles. After three crossing changes of the 2-handle's attaching curve, we obtain the second diagram in Figure~\ref{fig:pi1}. The rightmost diagram, obtained by further isotopy, shows that the modified 2-handle  can be cancelled with a 1-handle. This leaves a single 0-handle and 1-handle, representing $S^1 \times B^3$, which has $\pi_1 \cong \zz$. It follows that $\pi_1(B^4 \setminus D) \cong \zz$.

\begin{figure}[b]\center
	\def\svgwidth{\linewidth}
\begingroup%
  \makeatletter%
  \providecommand\color[2][]{%
    \errmessage{(Inkscape) Color is used for the text in Inkscape, but the package 'color.sty' is not loaded}%
    \renewcommand\color[2][]{}%
  }%
  \providecommand\transparent[1]{%
    \errmessage{(Inkscape) Transparency is used (non-zero) for the text in Inkscape, but the package 'transparent.sty' is not loaded}%
    \renewcommand\transparent[1]{}%
  }%
  \providecommand\rotatebox[2]{#2}%
  \newcommand*\fsize{\dimexpr\f@size pt\relax}%
  \newcommand*\lineheight[1]{\fontsize{\fsize}{#1\fsize}\selectfont}%
  \ifx\svgwidth\undefined%
    \setlength{\unitlength}{1047.90668638bp}%
    \ifx\svgscale\undefined%
      \relax%
    \else%
      \setlength{\unitlength}{\unitlength * \real{\svgscale}}%
    \fi%
  \else%
    \setlength{\unitlength}{\svgwidth}%
  \fi%
  \global\let\svgwidth\undefined%
  \global\let\svgscale\undefined%
  \makeatother%
  \begin{picture}(1,0.19861938)%
    \lineheight{1}%
    \setlength\tabcolsep{0pt}%
    \put(0,0){\includegraphics[width=\unitlength,page=1]{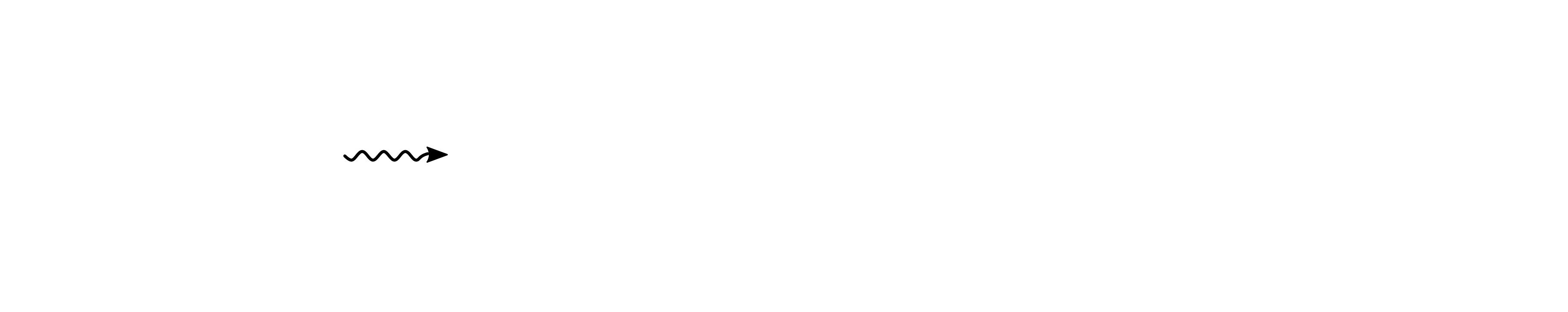}}%
    \put(0.53133795,0.09081762){\makebox(0,0)[lt]{\lineheight{1.25}\smash{\begin{tabular}[t]{l}$=$\end{tabular}}}}%
    \put(-0.00216671,0.00796895){\color[rgb]{0,0.50196078,0}\makebox(0,0)[lt]{\lineheight{1.25}\smash{\begin{tabular}[t]{l}$0$\end{tabular}}}}%
    \put(0.30128432,0.00796895){\color[rgb]{0,0.50196078,0}\makebox(0,0)[lt]{\lineheight{1.25}\smash{\begin{tabular}[t]{l}$0$\end{tabular}}}}%
    \put(0.57425412,0.02296232){\color[rgb]{0,0.50196078,0}\makebox(0,0)[lt]{\lineheight{1.25}\smash{\begin{tabular}[t]{l}$0$\end{tabular}}}}%
    \put(0.78183707,0.09081762){\makebox(0,0)[lt]{\lineheight{1.25}\smash{\begin{tabular}[t]{l}$=$\end{tabular}}}}%
    \put(0,0){\includegraphics[width=\unitlength,page=2]{images/pi1.pdf}}%
    \put(0.81759667,0.02296232){\color[rgb]{0,0.50196078,0}\makebox(0,0)[lt]{\lineheight{1.25}\smash{\begin{tabular}[t]{l}$0$\end{tabular}}}}%
  \end{picture}%
\endgroup%
	\caption{The exterior of $D$ has the homotopy type of $S^1$, as shown by performing a homotopy of the 2-handle's attaching curve followed by isotopy of the modified diagram.}\label{fig:pi1}
\end{figure}

A handle diagram for the exterior of $D'$ is obtained from that of $D$ by applying a 180$^\circ$ rotation through a vertical line, so an analogous argument shows $\pi_1(B^4 \setminus D') \cong \zz$.
\end{proof}

Combining Proposition~\ref{prop:top} with Theorem~\ref{thm:diff} immediately yields the following. 

\begin{cor}\label{cor:main}
The slice disks $D$ and $D'$ are exotically knotted rel boundary. 
\end{cor}

To establish the stronger conclusions of Theorem~\ref{thm:main}, we wish to distinguish surfaces in the 4-ball up to arbitrary ambient isotopy (and not merely isotopy rel boundary). Fortunately, if a knot $K$ has no nontrivial symmetries, then an ambient isotopy between surfaces bounded by $K$  can be promoted to an ambient isotopy rel boundary. To make this  precise,  let $\Diff(S^3,K)$ denote the group of diffeomorphisms of $S^3$ that fix $K$ setwise. The \emph{symmetry group} of a knot $K$ in $S^3$, denoted $\Sym(K)$, is the quotient of the group $\Diff(S^3,K)$ by the normal subgroup of diffeomorphisms that are isotopic to the identity through diffeomorphisms of the pair $(S^3,K)$.

%

\begin{lem}\label{lem:asym}
Let $K$ be a knot in $S^3$ with trivial symmetry group $\Sym(K) = \{\id\}$. If $K$ bounds properly embedded surfaces $\Sigma$ and $\Sigma'$ in $B^4$ that are ambiently isotopic, then $\Sigma$ and $\Sigma'$ are also ambiently isotopic rel boundary.
\end{lem}

This lemma follows from a more general but more technical result (Proposition~\ref{prop:dream}) that we prove in \S\ref{subsec:asym}. With these preliminaries in hand, Theorem~\ref{thm:main} follows quickly.


%
%
%
%
%

\begin{proof}[Proof of Theorem~\ref{thm:main}]
Consider again the knots $J_{m,n}$ with $n \geq 0$ from Example~\ref{ex:higher-genus}. We showed that the knot $J_{m,n}$ bounds a pair of smooth, oriented, properly embedded surfaces $\Sigma_{m,n}$ and $\Sigma'_{m,n}$ of genus $n$ in $B^4$ that induce distinct maps on Khovanov homology, hence are not smoothly isotopic rel boundary.
%



Observe that the surface $\Sigma'_{m,n}$ is obtained from $\Sigma_{m,n}$ by replacing the disk $D \subset \Sigma_{m,n}$ with the disk $D'$. (In particular, the disks $D_m$ and $D'_m$ are obtained by extending $D$ and $D'$ by a fixed concordance from $J$ to $J_m$.) Since $D$ and $D'$ are topologically isotopic rel boundary by Proposition~\ref{prop:top}, we conclude that $\Sigma_{m,n}$ and $\Sigma'_{m,n}$ are topologically isotopic rel boundary.

Finally, we address the stronger conclusion in the theorem. Using SnapPy \cite{snappy} inside Sage \cite{sagemath}, we verify that $\Sym(J_{m,n})$ is trivial if $m \gg0$; see \S\ref{subsec:asym}. By Lemma~\ref{lem:asym}, we  conclude that there is no smooth isotopy of $B^4$ carrying $\Sigma_{m,n}$ to $\Sigma'_{m,n}$ for  $m \gg 0$. (We also verified the claim for $m=0$, and we expect the claim to hold for all $m$.)
\end{proof}


\section{A braid-theoretic approach}
\label{sec:plam}


In this section, we develop braid-theoretic techniques for studying the cobordism maps in Khovanov homology. Our starting point is Plamenevskaya's invariant \cite{plamenevskaya:transverse-Kh}; in \S\ref{subsec:plam}, we review Plamenevskaya's construction and prove Theorem~\ref{thm:plam},  establishing the behavior of this invariant under a flexible class of link cobordisms that generalize complex curves.   In \S\ref{subsec:factorizations}, we review Rudolph's framework of braided surfaces and band factorizations \cite{rudolph:braided-surface}, which guides the construction of more refined classes in Khovanov homology that we can use to distinguish pairs of surfaces.

\subsection{Functoriality of Plamenevskaya's invariant.}\label{subsec:plam}
Given a transverse link $L$ in the standard contact $S^3$, Plamenevskaya defines a class $\psi(L) \in \Kh(L)$ that is invariant (up to sign) under isotopies through transverse links \cite{plamenevskaya:transverse-Kh}. Her construction leverages the correspondence between transverse links up to transverse isotopy and closed braids up to braid isotopy and positive Markov stabilization \cite{bennequin,os:markov,wrinkle}. (For  more background on transverse links, see \cite{etnyre:knot-intro}.)


To define $\psi(L)$, choose an $n$-stranded braid $\beta$ representing $L$. Consider the ``braided'' smoothing of the diagram into $n$ concentric circles by taking the oriented resolution at each crossing (i.e., 0-resolution at each positive crossing and 1-resolution at each negative crossing), and label each circle with an $x$. Plamenevskaya shows this is a cycle in bigrading $(h,q)=(0,w-n)$, where $w$ is the writhe of $\beta$. To prove the resulting class $\psi(L)$ defines a transverse link invariant, she shows $\psi$ is preserved by braid isotopy and positive Markov stabilization. For later use, we  note that $\psi$ is also preserved by  \emph{positive crossing resolutions},  which are simple saddle cobordisms that correspond to deleting a positive crossing $\sigma_i \in \beta$  \cite[Theorem 4]{plamenevskaya:transverse-Kh}.



Next, we recall some definitions and background on ascending surfaces from \cite{bo:qp,hayden:stein}. A smooth, oriented link cobordism $\Sigma \subset S^3 \times [0,1]$ is \emph{ascending} if the projection $\rho: S^3 \times [0,1] \to [0,1]$ restricts to a Morse function on $\Sigma$ and, except at critical points of $\rho|_\Sigma$, the level sets of $\rho|_\Sigma$ are transverse to the standard contact structure on $S^3 \times \{t\}$. At each critical point  $p \in \Sigma$ of $\rho|_\Sigma$, the tangent plane $T_p \Sigma$ coincides with the contact plane $\xi_p$. The critical point is said to be \emph{positive} or \emph{negative} according to whether the orientations on $T_p \Sigma$ and $\xi_p$ agree or disagree, respectively.

To prove that the transverse invariant $\psi$ behaves well with respect to ascending cobordisms with positive critical points, we leverage a relationship between ascending surfaces and braids. The following result can be extracted from the proofs of Lemma 3.5 and Theorem 4.3 in \cite{hayden:stein}; see \cite[\S7.5]{jmz:exotic} for a similar application of \cite{hayden:stein}.

\begin{thm}[\cite{hayden:stein}]\label{thm:braid}
Let $\Sigma \subset S^3 \times [0,1]$ be an ascending cobordism with positive critical points, and suppose that $\Sigma$ has a single critical point at height $t$. After an isotopy (through ascending cobordisms) supported in a small neighborhood $S^3 \times [t-\epsilon,t+\epsilon]$, we may assume that the regular level sets of $\Sigma$ near $S^3 \times \{t\}$ are braided. Morever, the subcobordism $\Sigma \cap S^3 \times [t-\epsilon/2,t+\epsilon/2]$ between the braids $\beta_\pm = \Sigma \cap S^3 \times \{t \pm \epsilon/2\}$ corresponds to either a braided birth or a braided saddle move (with a right-handed half-twist) as depicted in Figure~\ref{fig:braid-moves}.
\end{thm}

\begin{figure}[h]\center
	\def\svgwidth{\linewidth}
\begingroup%
  \makeatletter%
  \providecommand\color[2][]{%
    \errmessage{(Inkscape) Color is used for the text in Inkscape, but the package 'color.sty' is not loaded}%
    \renewcommand\color[2][]{}%
  }%
  \providecommand\transparent[1]{%
    \errmessage{(Inkscape) Transparency is used (non-zero) for the text in Inkscape, but the package 'transparent.sty' is not loaded}%
    \renewcommand\transparent[1]{}%
  }%
  \providecommand\rotatebox[2]{#2}%
  \newcommand*\fsize{\dimexpr\f@size pt\relax}%
  \newcommand*\lineheight[1]{\fontsize{\fsize}{#1\fsize}\selectfont}%
  \ifx\svgwidth\undefined%
    \setlength{\unitlength}{937.11126709bp}%
    \ifx\svgscale\undefined%
      \relax%
    \else%
      \setlength{\unitlength}{\unitlength * \real{\svgscale}}%
    \fi%
  \else%
    \setlength{\unitlength}{\svgwidth}%
  \fi%
  \global\let\svgwidth\undefined%
  \global\let\svgscale\undefined%
  \makeatother%
  \begin{picture}(1,0.19888046)%
    \lineheight{1}%
    \setlength\tabcolsep{0pt}%
    \put(0,0){\includegraphics[width=\unitlength,page=1]{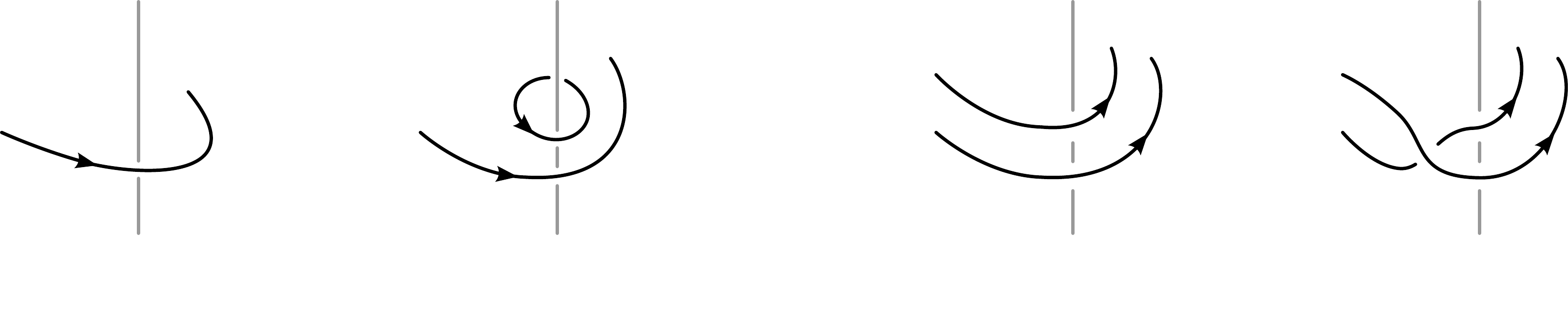}}%
    \put(0.11267902,0.06089194){\color[rgb]{0,0,0}\makebox(0,0)[lt]{\lineheight{1.25}\smash{\begin{tabular}[t]{l}$\beta_-$\end{tabular}}}}%
    \put(0,0){\includegraphics[width=\unitlength,page=2]{images/operations.pdf}}%
    \put(0.19501157,0.00644017){\color[rgb]{0,0,0}\makebox(0,0)[lt]{\lineheight{1.25}\smash{\begin{tabular}[t]{l}(a)\end{tabular}}}}%
    \put(0.79392968,0.00644017){\makebox(0,0)[lt]{\lineheight{1.25}\smash{\begin{tabular}[t]{l}(b)\end{tabular}}}}%
    \put(0.38319119,0.06089194){\color[rgb]{0,0,0}\makebox(0,0)[lt]{\lineheight{1.25}\smash{\begin{tabular}[t]{l}$\beta_+$\end{tabular}}}}%
    \put(0.72093125,0.06089194){\color[rgb]{0,0,0}\makebox(0,0)[lt]{\lineheight{1.25}\smash{\begin{tabular}[t]{l}$\beta_-$\end{tabular}}}}%
    \put(0.98984275,0.06249261){\color[rgb]{0,0,0}\makebox(0,0)[lt]{\lineheight{1.25}\smash{\begin{tabular}[t]{l}$\beta_+$\end{tabular}}}}%
  \end{picture}%
\endgroup%

	\vspace{-10pt}
	\caption{The subcobordisms associated to a braided (a) birth and (b) positive saddle.}\label{fig:braid-moves}
\end{figure}

\begin{proof}[Proof of Theorem~\ref{thm:plam}]
Let $\Sigma \subset S^3 \times [0,1]$ be an ascending cobordism with positive critical points, viewed as a cobordism from $L_1$ to $L_0$. We may perturb $\Sigma$ (using a small isotopy rel boundary through ascending surfaces) to ensure that each critical level set contains a single critical point. Moreover, by a further isotopy rel boundary, we may assume that $\Sigma$ has the braided structure  from Theorem~\ref{thm:braid} near each critical level set. 

By subdividing $\Sigma$ and composing the associated cobordism maps, it suffices to consider three cases. First, between critical level sets, $\Sigma$ is a concordance swept out by a transverse isotopy between transverse links. In this case, $\Kh(\Sigma)$ preserves the transverse invariant (up to sign) by the proof of \cite[Theorem 2]{plamenevskaya:transverse-Kh}.

Next we consider $\Sigma$ near critical level sets, keeping in mind that we are viewing it ``in reverse'' as a cobordism from $L_1$ to $L_0$. The two remaining cases to consider are the cobordisms going from $\beta_+$ to $\beta_-$ in parts (a) and (b) of Figure~\ref{fig:braid-moves}. In part (a),  the cobordism from $\beta_+$ to $\beta_-$ is a Morse death. In the Khovanov chain complexes associated to these braided diagrams, the chains representing  $\psi(\beta_-)$ and $\psi(\beta_+)$ agree except for an $x$-labeled circle in the latter that corresponds to the unknotted component that is killed by the Morse death. The associated cobordism map is determined by applying the map $\varepsilon$ to this distinguished $x$-labeled circle, hence the induced map takes $\psi(\beta_+)$ to $\psi(\beta_-)$. In part (b) of Figure~\ref{fig:braid-moves}, the cobordism from $\beta_+$ to $\beta_-$ is a positive crossing resolution, which is shown to take $\psi(\beta_+)$ to $\psi(\beta_-)$ in \cite[Theorem 4]{plamenevskaya:transverse-Kh}.
\end{proof}

\subsection{Braided surfaces and Khovanov homology.}\label{subsec:factorizations}
Motivated by the above connections between braids and  Khovanov homology, we recall the framework of braided surfaces (\S\ref{subsubsec:braided}) and use this to develop further computational tools (\S\ref{subsubsec:braided-cob}). In what follows, $B_n$ denotes the $n$-stranded braid group, and we will often use $\beta$ to denote both an element of $B_n$ and the link in $S^3$ obtained as its closure.

\subsubsection{Band factorizations and braided surfaces.} \label{subsubsec:braided}
Following Rudolph \cite{rudolph:braided-surface}, a \emph{positive band} (resp.,~\emph{negative band}) in $B_n$ is a word of the form $w \sigma_i w^{-1}$ (resp.,~$w \sigma_i^{-1} w^{-1}$), where $\sigma_i$ is a standard positive Artin generator and $w$ is any word in $B_n$. Any factorization of a braid $\beta \in B_n$ as a product of bands is called a \textit{band factorization} and determines a ribbon-immersed surface in $S^3$ obtained from a collection of $n$ parallel disks by attaching a half-twisted band for each term $w \sigma_i^{\pm1} w^{-1}$; see Figures~\ref{fig:braided-surfaces} and \ref{fig:10-148} for examples. Pushing the interior of the surface into $B^4$ yields a \emph{braided surface} that is smooth and properly embedded in $B^4$ 
with boundary
the closure of the braid $\beta$. 

\begin{figure}\center
\includegraphics[width=\linewidth]{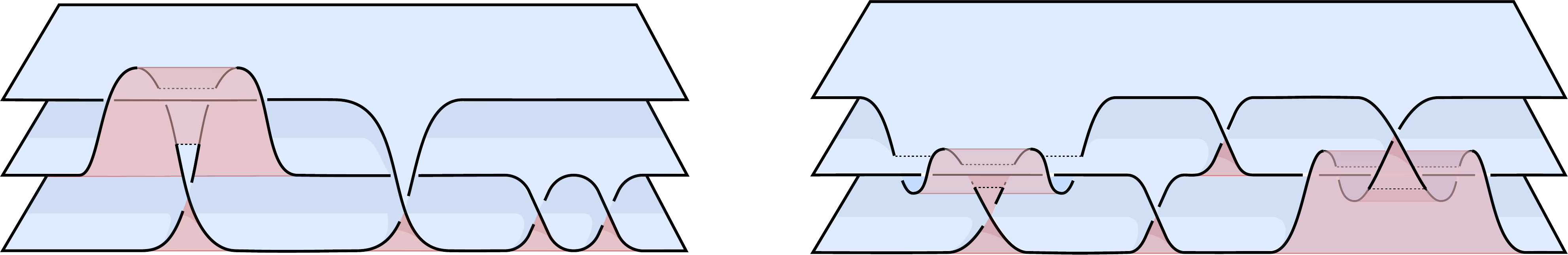}
\caption{Positively braided surfaces corresponding to the quasipositive braid words $(\sigma_1^{-2}\sigma_2 \sigma_1^2)(\sigma_1 \sigma_2 \sigma_1^{-1}) \sigma_2^2$ on the left and $(\sigma_1^{-3} \sigma_2 \sigma_1^3)\sigma_2\sigma_1 (\sigma_2^{-3} \sigma_1 \sigma_2^3)$ on the right.}
\label{fig:10-148}
\end{figure}

In \cite{rudolph:braided-surface}, Rudolph showed that any properly embedded ribbon surface in $B^4$ (i.e., one that has no local maxima) is isotopic to a braided surface.  Many examples of inequivalent surfaces with the same boundary  can be expressed using different band factorizations of the same braid group element. (Indeed, by combining Rudolph's work with Markov's theorem \cite{markov}, one can show that any pair of ribbon surfaces with isotopic boundary can be related this way.)  

\begin{ex} \label{ex:10-148} The positively braided surfaces in Figure~\ref{fig:10-148} are both bounded by the knot $10_{148}$. Below, we relate these braid words directly using braid group relations, including $\sigma_i \sigma_{i+1} \sigma_{i}^{-1}=\sigma_{i+1}^{-1}\sigma_{i} \sigma_{i+1}$. (The underlined terms are marked for later use.)
\begin{align*}
(\sigma_1^{-3} \underline{\sigma_2} \sigma_1^3)\sigma_2\sigma_1 (\sigma_2^{-3} \sigma_1 \sigma_2^3)
 &= 
 (\sigma_1^{-3} \underline{\sigma_2} \sigma_1^3)\sigma_2\sigma_1 \sigma_2^{-2} (\sigma_2^{-1} \sigma_1 \sigma_2)\sigma_2^2
 \\
  &= (\sigma_1^{-3} \underline{\sigma_2} \sigma_1^3)\sigma_2\sigma_1 \sigma_2^{-2} (\sigma_1 \sigma_2 \sigma_1^{-1})\sigma_2^2
   \\
  &= (\sigma_1^{-2} (\sigma_1^{-1} \underline{\sigma_2} \sigma_1) \sigma_1^2)\sigma_2\sigma_1 \sigma_2^{-2} (\sigma_1 \sigma_2 \sigma_1^{-1})\sigma_2^2
     \\
  &= (\sigma_1^{-2} (\sigma_2 \underline{\sigma_1} \sigma_2^{-1}) \sigma_1^2)\sigma_2\sigma_1 \sigma_2^{-2} (\sigma_1 \sigma_2 \sigma_1^{-1})\sigma_2^2
       \\
  &= \sigma_1^{-2} \sigma_2 \underline{\sigma_1} (\sigma_2^{-1}\sigma_1^2 \sigma_2)\sigma_1 \sigma_2^{-2} (\sigma_1 \sigma_2 \sigma_1^{-1})\sigma_2^2
         \\
  &= \sigma_1^{-2} \sigma_2 \underline{\sigma_1} (\sigma_1 \sigma_2^2 \sigma_1^{-1})\sigma_1 \sigma_2^{-2} (\sigma_1 \sigma_2 \sigma_1^{-1})\sigma_2^2
           \\
  &= (\sigma_1^{-2} \sigma_2 \underline{\sigma_1} \sigma_1) \sigma_2^2 \sigma_1^{-1}\sigma_1 \sigma_2^{-2} (\sigma_1 \sigma_2 \sigma_1^{-1})\sigma_2^2
             \\
  &= (\sigma_1^{-2} \sigma_2 \underline{\sigma_1} \sigma_1)(\sigma_1 \sigma_2 \sigma_1^{-1})\sigma_2^2
\end{align*}

\end{ex}

\begin{rem} Work of Rudolph \cite{rudolph:qp-alg} and Boileau-Orevkov \cite{bo:qp} shows that a surface in $B^4$ is isotopic to a \emph{positively} braided surface (i.e., with only positive bands)  if and only if it is isotopic to  the intersection of a smooth complex curve with $B^4 \subset \cc^2$. \end{rem}

\subsubsection{Cobordism maps induced by braided surfaces.}\label{subsubsec:braided-cob}

We highlight an elementary lemma that simplifies calculations of cobordism maps induced by braided surfaces. Its proof is a simple exercise using Table~\ref{table_reidemeister_redux}; see Figure~\ref{fig:core-res} as well.

\begin{lem}\label{lem:res}
The chain map induced by an oriented crossing resolution (of either sign) sends any disoriented smoothing to zero. On oriented smoothings, it acts as the identity if the crossing is positive and as $\pm\, \tfrac{1}{2} \left(\ \zdott \, - \, \zdotb\ \right)$ if the crossing is negative.
\end{lem}

\begin{figure}\center
       \labellist
\pinlabel {\scriptsize$0$} at 365 31
\endlabellist
	\includegraphics[width=\linewidth]{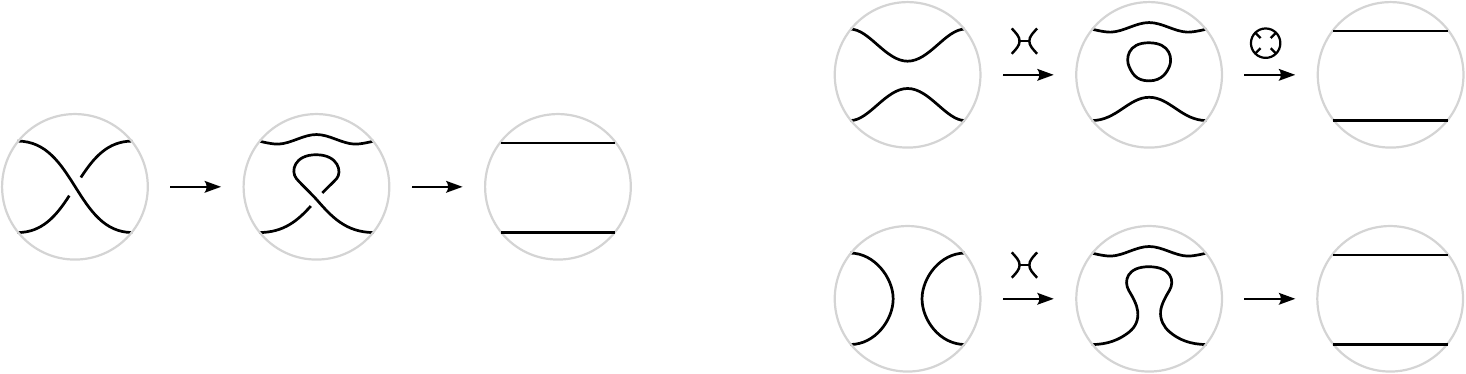}
		\caption{(Left) A positive crossing resolution in a braid as a composition of a saddle move and Reidemeister I move. (Right) The induced map on   smoothings.}
	\label{fig:core-res}
\end{figure}

Before formalizing our approach, we give an example illustrating the core ideas.

\begin{prop}\label{prop:10-148}
Let $\Sigma$ and $\Sigma'$ denote the positively braided, genus-1 surfaces associated to the band factorizations $(\sigma_1^{-2}\sigma_2 \sigma_1^2)(\sigma_1 \sigma_2 \sigma_1^{-1}) \sigma_2^2$  and $(\sigma_1^{-3} \sigma_2 \sigma_1^3)\sigma_2\sigma_1 (\sigma_2^{-3} \sigma_1 \sigma_2^3)$, respectively, for the knot $10_{148}$. There is no smooth isotopy of $B^4$ carrying $\Sigma$ to $\Sigma'$.
\end{prop}

\begin{proof}  Let $\beta$ and $\beta'$ denote the braid diagrams for $10_{148}$ corresponding to the braid factorizations underlying $\Sigma$ and $\Sigma'$. The cobordism $\Sigma: \beta \to \emptyset$ is described by the sequence of diagrams on the left side of Figure~\ref{fig:10-148-cycle}. The right side of Figure~\ref{fig:10-148-cycle} begins with labeled smoothing $\phi$ that is easily checked to be a cycle, then tracks it through the cobordism map using Lemma~\ref{lem:res} and Table~\ref{table_reidemeister_redux} to conclude $\CKh(\Sigma)=-1$.

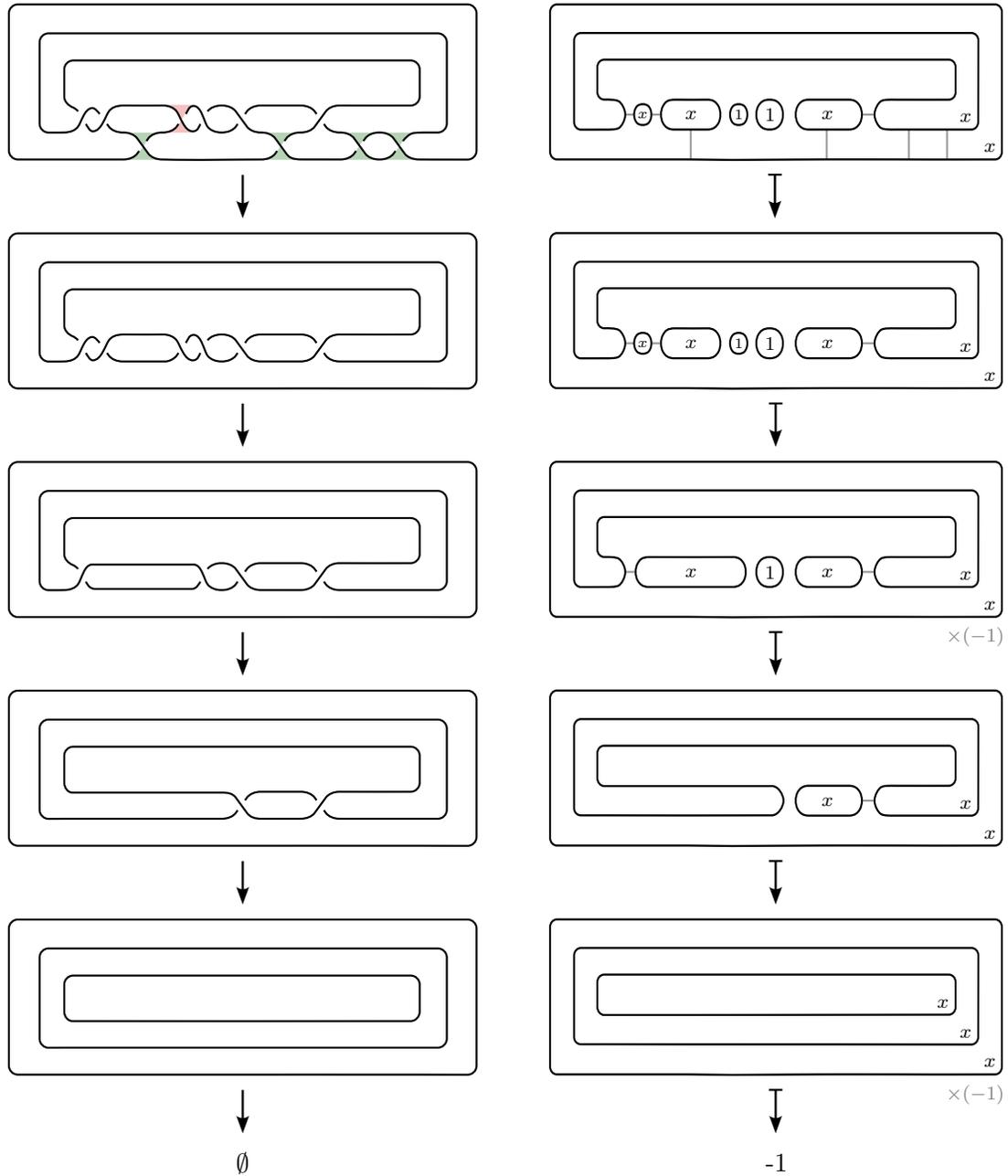
\begin{figure}
\center
\labellist
\pinlabel {\tiny $x$} <0pt,37.75pt>  at 1378.5 2089
\pinlabel {\scriptsize $x$} <0pt,37.75pt>  at 1483 2089
\pinlabel {\tiny $1$} <0pt,37.75pt>  at 1587 2090
\pinlabel {\scriptsize $1$} <0pt,37.75pt>  at 1654.5 2090
\pinlabel {\scriptsize $x$} <0pt,37.75pt>  at 1779 2089
\pinlabel {\scriptsize $x$} <0pt,37.75pt>  at 2079 2082
\pinlabel {\scriptsize $x$} <0pt,37.75pt>  at 2132 2018

\pinlabel {\tiny $x$} <0pt,37.75pt>  at 1378.5 1591
\pinlabel {\scriptsize $x$} <0pt,37.75pt>  at 1483 1591
\pinlabel {\tiny $1$} <0pt,37.75pt>  at 1587 1592
\pinlabel {\scriptsize $1$} <0pt,37.75pt>  at 1654.5 1592
\pinlabel {\scriptsize $x$} <0pt,37.75pt>  at 1779 1591
\pinlabel {\scriptsize $x$} <0pt,37.75pt>  at 2079 1584
\pinlabel {\scriptsize $x$} <0pt,37.75pt>  at 2132 1520

\pinlabel {\scriptsize $x$} <0pt,37.75pt>  at 1483 1094
\pinlabel {\scriptsize $1$} <0pt,37.75pt>  at 1654.5 1095
\pinlabel {\scriptsize $x$} <0pt,37.75pt>  at 1779 1094
\pinlabel {\scriptsize $x$} <0pt,37.75pt>  at 2079 1087
\pinlabel {\scriptsize $x$} <0pt,37.75pt>  at 2132 1023
\pinlabel {\scriptsize \textcolor{gray}{$\times(-1)$}} <0pt,37.75pt>  at 2100 955

\pinlabel {\scriptsize $x$} <0pt,37.75pt>  at 1779 597
\pinlabel {\scriptsize $x$} <0pt,37.75pt>  at 2079 590
\pinlabel {\scriptsize $x$} <0pt,37.75pt>  at 2132 526

\pinlabel {\scriptsize $x$} <0pt,37.75pt>  at 2030 158
\pinlabel {\scriptsize $x$} <0pt,37.75pt>  at 2079 93
\pinlabel {\scriptsize $x$} <0pt,37.75pt>  at 2132 29
\pinlabel {\scriptsize \textcolor{gray}{$\times(-1)$}} <0pt,37.75pt>  at 2100 -40

\pinlabel {$\emptyset$} at 510 10
\pinlabel {-1}  at 1669 10

\endlabellist

\includegraphics[width=\linewidth]{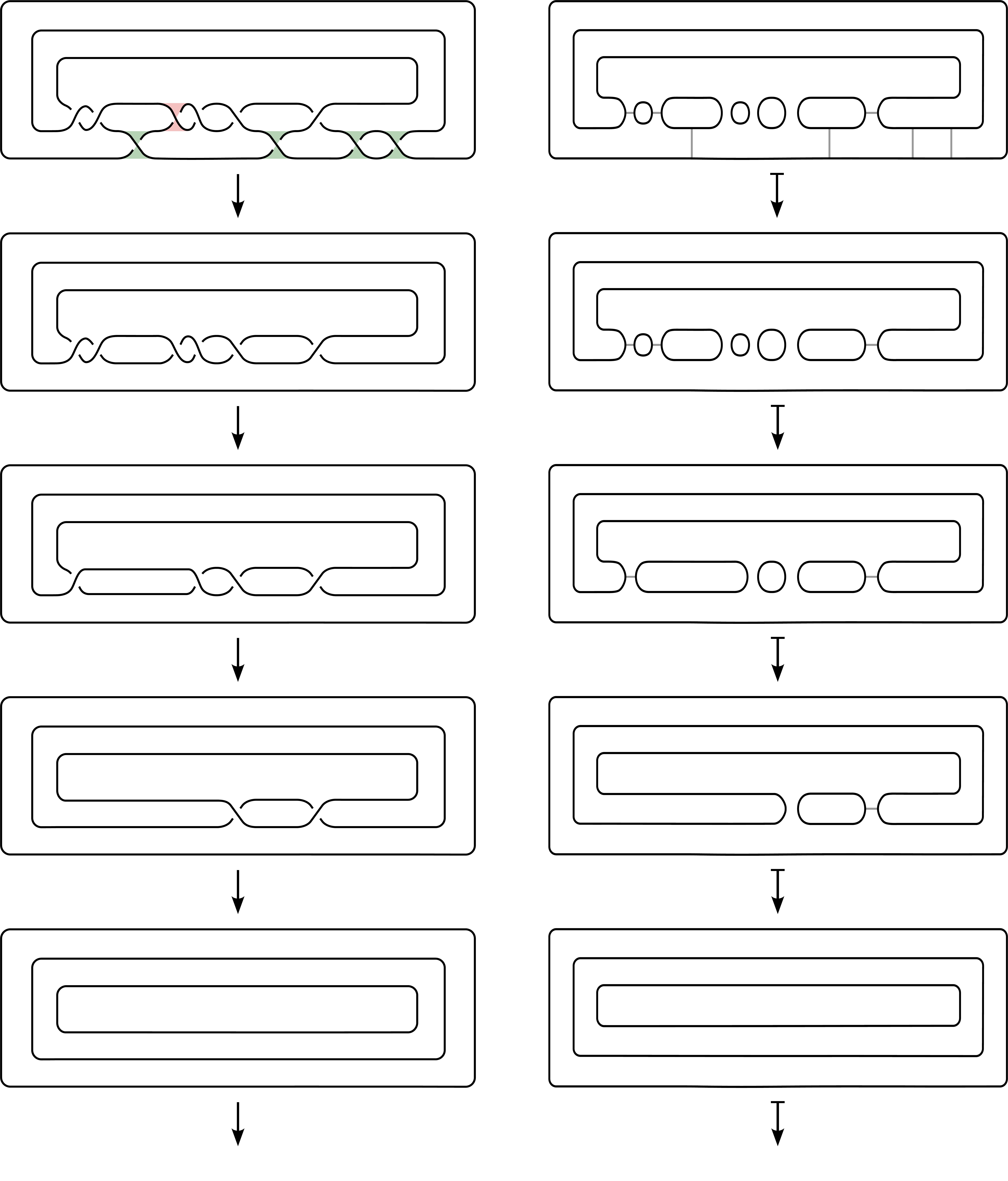}

\vspace{5pt}
\caption{A sequence of diagrams illustrating key steps in the cobordism $\Sigma: \beta \to \emptyset$ and the behavior of a cycle $\phi \in \CKh(\beta)$ under the induced map $\CKh(\Sigma)$. The cobordism begins with four positive crossing resolutions, followed by three Reidemeister II moves, and ends with three Morse deaths.}
\label{fig:10-148-cycle}
\end{figure}

We cannot directly compare $\CKh(\Sigma)$ and $\CKh(\Sigma')$ because they have different domains $\CKh(\beta)$ and $\CKh(\beta')$. However, if we take the isotopy from $\beta'$ to $\beta$ exhibited in Example~\ref{ex:10-148} and smoothly extend it over $B^4$, we carry $\Sigma'$ to an isotopic surface $\Sigma''$ bounded by $\beta$. Moreover, by  taking the saddle moves on $\beta'$ that correspond to $\Sigma'$ and tracking them through the braid isotopy, we can locate saddle moves on $\beta$ that correspond to $\Sigma''$. We do this for a chosen positive crossing using the underlined terms in Example~\ref{ex:10-148}. The translated saddle move is represented by the red band in Figure~\ref{fig:10-148-cycle}. The cobordism $\Sigma'':\beta \to \emptyset$ begins by resolving this crossing, which we note has been given the \emph{disoriented} smoothing in the cycle $\phi$. It follows that this positive crossing resolution cobordism kills $\phi$ by Lemma~\ref{lem:res}, hence $\CKh(\Sigma')(\phi)=0$.

It follows that $\Sigma$ and $\Sigma''$ are not smoothly isotopic rel boundary. Moreover, a direct calculation in SnapPy \cite{snappy} shows that the knot $10_{148}$ has trivial symmetry group, so Lemma~\ref{lem:asym} implies that there is no smooth isotopy of $B^4$ carrying $\Sigma''$ to $\Sigma$. It follows that the same is true of $\Sigma'$ and $\Sigma$, since $\Sigma'$ is smoothly isotopic to $\Sigma''$.
\end{proof}



\clearpage

Let us formalize some of the ideas seen in the above proof. We say that a smoothing of a braid $\beta$ is \emph{compatible} with a given band factorization of $\beta$ if, for each band $w \sigma_i^{\pm1} w^{-1}$,
\begin{enumerate}[label=\emph{(\roman*)}]
\item the core crossing $\sigma_i^{\pm1}$ is given the oriented smoothing (i.e., 0-smoothing for $\sigma_i$ and 1-smoothing for $\sigma_i^{-1}$), and
\item for each crossing in $w$ and  corresponding inverse crossing in $w^{-1}$, either both have oriented smoothings or both have disoriented smoothings.\end{enumerate}
We also say that a labeling $\alpha \in \CKh(\beta)$ of such a smoothing is  \emph{compatible} with the given band factorization of $\beta$. The cycle underlying Plamenevskaya's invariant is a prototypical example, and it is compatible with every band factorization because each crossing is given the oriented smoothing. For an example of a compatible cycle that contrasts with Plamenevskaya's cycle, the cycle in Figure~\ref{fig:10-148-cycle} from the proof of Proposition~\ref{prop:10-148} has disoriented smoothings on all conjugating crossings in the band factorization. 

\begin{lem} \label{lem:incompatible}
If $\Sigma$ is a braided surface given by a band factorization of $\beta$ and $\alpha \in \CKh(\beta)$ is an incompatible labeled smoothing, then $\CKh(\Sigma)(\alpha)=0$.

\end{lem}

\begin{proof} The cobordism $\Sigma: \beta\to \emptyset$ naturally begins with a sequence of crossing resolutions,  each resolving the core crossing of a band $w \sigma_i^{\pm1} w^{-1}$. After passing these saddles, we have a diagram for an unlink given by a product of braid words of the form $ww^{-1}$. This is simplified to the trivial braid by a sequence of Reidemeister II moves, each canceling the final crossing in a word $w$ with the first crossing in its inverse  $w^{-1}$. The cobordism ends with Morse deaths deleting the components of the trivial braid.

Suppose $\alpha$ is incompatible with the braid factorization. If one of the core crossings in a band is given the disoriented smoothing, then  Lemma~\ref{lem:res} says that $\alpha$ is killed by one of the initial crossing resolutions, so $\CKh(\Sigma)(\alpha)=0$.

Next consider the case where there is a pair of corresponding inverse crossings in some $w$ and $w^{-1}$ such that one crossing is given the oriented smoothing and the other is given the disoriented smoothing. We  proceed in the cobordism until these crossings are adjacent and are ready to be canceled with a Reidemeister II move. (Note that, after each previous Reidemeister II move, the diagram and underlying smoothing are unchanged away from the two 
 crossings being canceled --- however, the labelings and connectivity of the loops may change). At the stage where we cancel the two crossings in question,  the cobordism has the local form shown in the bottom-left corner of  Table~\ref{table_reidemeister_redux} (or its mirror). The smoothings and chain map are locally given by the bottom two rows on the right side of the table, and these are both zero maps, so $\CKh(\Sigma)(\alpha)=0$.
\end{proof}

The preceding observations help us identify chain elements in the kernel of a braided surface's cobordism map. But a similar perspective can also help identify elements in the support of the map, especially when considering \textsl{positively} braided surfaces. This provides the following heuristic, which we demonstrate in Proposition~\ref{prop:braided-disks} below. 

\smallskip

\begin{heuristic}\label{heuristic}
Given a braided surface $\Sigma$  associated to a band factorization $\beta = \prod  w_k \sigma_{i_k}^{\pm 1} w_k^{-1}$, suppose $\Sigma'$ is another surface with $\partial \Sigma'=\beta$ whose movie begins by resolving a crossing $c$ in a subword $w_k$ or $w_k^{-1}$.
\begin{enumerate}[label=\arabic*.]
\item Identify a chain $\phi \in \CKh(\beta)$ that has a disoriented resolution at $c$ and satisfies $\CKh(\Sigma)(\phi)\neq 0$.

\item Search for a chain  $\alpha \in \ker \CKh(\Sigma) \cap \ker \CKh(\Sigma')$ with $\partial \alpha = \partial \phi$.
\end{enumerate}
\smallskip

The difference $\phi-\alpha$ represents a cycle $\delta \in \Kh(\beta)$ such that $\Kh(\Sigma)(\delta)\neq \Kh(\Sigma')(\delta)$.
\end{heuristic}

The first step of this strategy is often straightforward. In the second step, it is also easy to identify many elements of $\ker \CKh(\Sigma) \cap \ker \CKh(\Sigma')$; in light of Lemma~\ref{lem:incompatible}, it is natural to consider  the subcomplex of $\CKh(\beta)$ generated by incompatible smoothings. Also note that $\ker \CKh(\Sigma')$ contains the subcomplex of $\CKh(\beta)$ where the crossing $c$ is assigned a disoriented smoothing. 

 To see this in practice, we will give an alternative argument that distinguishes the disks from  Figure~\ref{fig:main-disks} via their braided representatives in Figure~\ref{fig:braided-surfaces}. For convenience, we now let $D$ and $D'$ denote these latter braided representatives. 


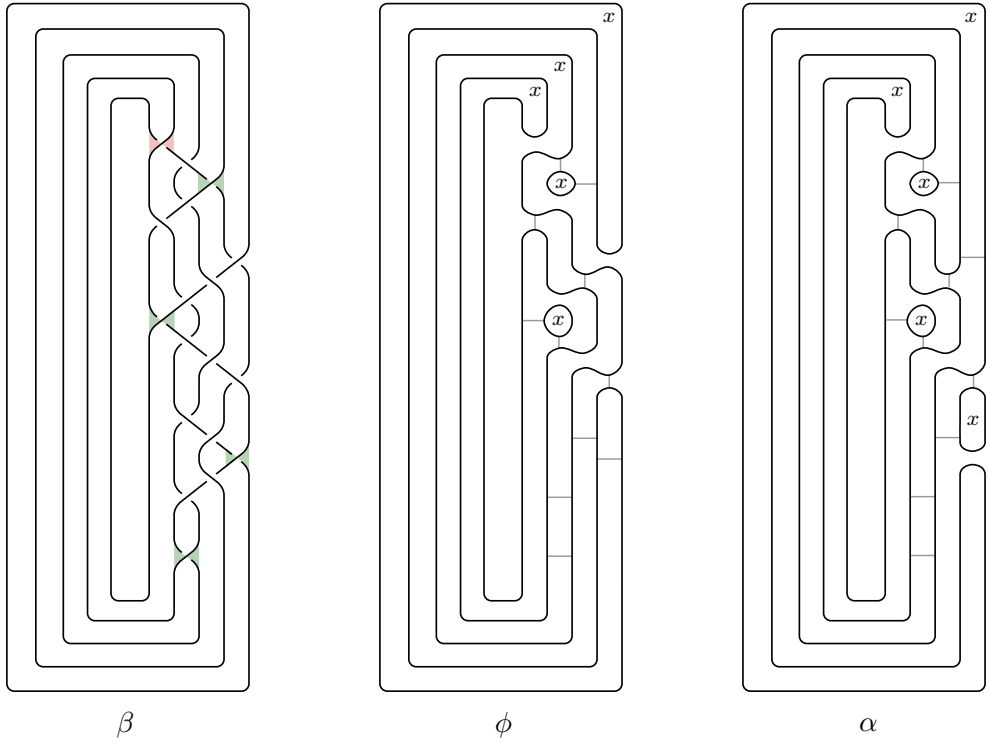
\begin{figure}\center
       \labellist
\scriptsize\hair 2pt
\pinlabel {$x$} at 1819 2035
\pinlabel {$x$} at 1671 1888
\pinlabel {$x$} at 1596.5 1813
\pinlabel {$x$} at 1675 1536
\pinlabel {$x$} at 1666 1120

\pinlabel {$x$} at 2910 2035
\pinlabel {$x$} at 2687.5 1813
\pinlabel {$x$} at 2769 1536
\pinlabel {$x$} at 2760 1120
\pinlabel {$x$} at 2916 820

\pinlabel {\normalsize $\beta$} at 360  -100
\pinlabel {\normalsize $\phi$} at 1500  -100
\pinlabel {\normalsize $\alpha$} at 2600  -100

\endlabellist
	\includegraphics[width=.9\linewidth]{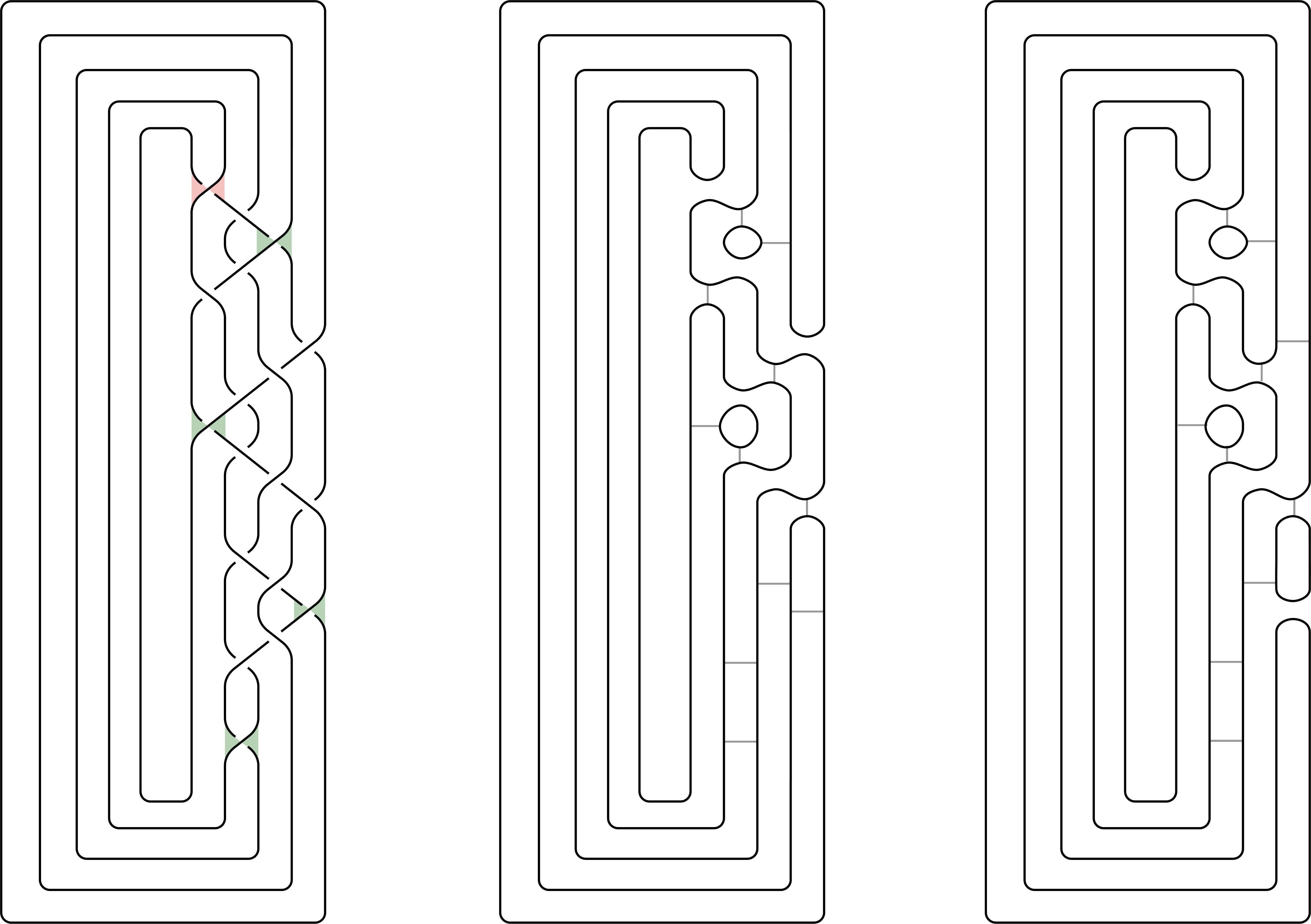}

	\vspace{.25in}
	
	\caption{A closed braid $\beta$ representing $J$, together with chain elements $\phi$ and $\alpha$ satisfying Heuristic \ref{heuristic}.}
	\label{fig:beta-phi}
	\end{figure}

\begin{prop}\label{prop:braided-disks}
There is a braid isotopy between the boundaries of the braided disks $D$ and $D'$ in Figure~\ref{fig:braided-surfaces} that does not extend to any smooth isotopy of $B^4$.
\end{prop}

\begin{proof}
 As shown in \cite[\S A.1]{hayden:curves}, there is a braid isotopy from $\partial D'$ to $\partial D$ that takes a band from $D'$ to a band corresponding to the final $\sigma_1$-crossing (highlighted in red) in Figure~\ref{fig:beta-phi}. For convenience, let $D''$ denote the image of $D'$ under any smooth extension of this isotopy. 
The chain element $\phi$ from Figure~\ref{fig:beta-phi} satisfies (1) in Heuristic~\ref{heuristic}. The boundary $\partial \phi$  consists of a single term (corresponding to changing the 0-resolution inside the outermost circle to a 1-resolution). Changing one of the other 1-resolutions in $\partial \phi$ to a 0-resolution yields the chain element $\alpha$ shown on the righthand side of Figure~\ref{fig:beta-phi}, which satisfies (2) in Heuristic~\ref{heuristic}. The difference of these chains represents a homology class that distinguishes the maps $\Kh(D)$ and $\Kh(D'')$, hence the disks $D$ and $D''$ are not smoothly isotopic rel boundary. It follows that the isotopy from $\partial D'$ to $\partial D$ cannot extend to a smooth isotopy of $B^4$.
\end{proof}

\titleformat{\section}{\large\bfseries}{}{0pt}{\center Appendix \thesection:  }
\titlespacing{\section}{0pt}{*4}{*1.5}
\setcounter{section}{0}
\renewcommand{\thesection}{\Alph{section}}\setcounter{subsection}{1}
\titleformat{\subsection}[runin]{\bfseries}{}{0pt}{\thesection.\arabic{subsection} \ \ }


\section{Isotopies and symmetry groups}

\subsection{Upgrading to isotopy rel boundary.} Under certain conditions, an  isotopy between surfaces with the same boundary can be upgraded to an isotopy rel boundary.

\begin{prop}\label{prop:dream}
Let $\Sigma_0$ and $\Sigma_1$ be properly embedded surfaces in $B^4$ bounded by the same  knot $K$ in $S^3$. Suppose there is an ambient isotopy of $B^4$ carrying $\Sigma_0$ to $\Sigma_1$, and let $f_1$ denote the induced diffeomorphism of the pair $(S^3,K)$ at time $t=1$. If $f_1$ is isotopic to the identity through diffeomorphisms of the pair $(S^3,K)$, then $\Sigma_0$ and $\Sigma_1$ are ambiently isotopic rel boundary.
\end{prop}

\begin{proof}
To begin, let $F_t$ with $t \in [0,1]$ denote the ambient isotopy of $B^4$ carrying $\Sigma_0$ to $\Sigma_1$, and let $f_t$ denote the induced isotopy of the boundary $S^3$.  We take a moment to make some simplifying assumptions.  First, we may assume that the ambient isotopy fixes a point that is disjoint from all the intermediate surfaces $\Sigma_t=F_t(\Sigma_0)$. (This is a straightforward application of the isotopy extension theorem\footnote{For example, see \cite[\S8]{hirsch}. While Hirsch's statements are phrased in terms of submanifolds lying entirely in either the ambient manifold's boundary or interior, the arguments carry over directly to our setting of properly embedded surfaces in $B^4$.}, which allows us to produce a modified isotopy  that agrees with $F_t$ on a neighborhood of $\partial B^4$ and $\cup_t \Sigma_t$ but is the identity outside a larger neighborhood of this subset.) It will be notationally convenient to remove this fixed point from $B^4$ and view its complement as $S^3 \times (-\infty,1]$. Since the surfaces $\Sigma_t$ are compact, they lie in a sufficiently large compact collar neighborhood of the boundary. For convenience, we will assume that they lie in $S^3 \times [-1,1]$ and that the ambient isotopy is supported inside $S^3 \times [-2,1]$. Finally, it will also be technically convenient to assume that $\Sigma_t$ intersects the collar neighborhood $S^3 \times [0,1]$ along the cylinder $f_t(K) \times [0,1]$.

{We begin by modifying the isotopy so that its time-1 map restricts to the identity on $\partial B^4 = S^3 \times \{1\}$.} By hypothesis, the diffeomorphism $f_1: S^3 \to S^3$ at time $t=1$ is isotopic to the identity through diffeomorphisms of the pair $(S^3,K)$. Let $g_t : (S^3,K) \to (S^3,K)$ be such an isotopy from $g_0=f_1$ back to $g_1 = \id$.  It is straightforward to extend $g_t$ to an isotopy $G_t :B^4 \to B^4$ such that $G_t(\Sigma_1)=\Sigma_1$ for all $t \in [0,1]$: Fix a small value $\epsilon>0$ and a smooth, monotone function $\rho: [0,1] \to [0,1]$ that equals $0$ on $[0,\epsilon]$ and equals $1$ on $[1-\epsilon,1]$. We can define $G_t$ by demanding that
\begin{enumerate}[label=(\roman*)]
\item $G_t$ agrees with $F_1$ outside $S^3 \times (0,1]$ for all $t \in [0,1]$ and
\item at each point $(x,s) \in S^3 \times [0,1]$, we have
$G_t(x,s) = \left(g_{t\rho(s)}(x),s\right).$
\end{enumerate}
Concatenating the isotopies $F_t$ and $G_t$ yields an isotopy $H_t$ of $B^4$ that still carries $\Sigma_0$ to $\Sigma_1$ and whose time-1 map $H_1$ restricts the identity on $\partial B^4$.

We will now modify the entire isotopy by ``wringing out'' the boundary isotopy $h_t$ of $S^3 \times \{1\}$, letting it run down over $\Sigma_t$. {Choose a smooth, monotone function $\mu: (-\infty,1]\to [0,1]$ that equals $0$ on $(-\infty,-2]$ and equals $1$ on $[-1,1]$, and set
$$I_t(x,s) = \left(h^{-1}_{t\mu(s)}(x),s\right).$$
Since $h_0^{-1}=\id$, the homotopy $I_t$ is supported on $S^3\times (-2,1] \subset B^4$ and thus we can further extend $I_t$ to $B^4$ so that it fixes  the point at infinity.}


We claim that $I_t \circ H_t$ is an isotopy of $B^4$ that fixes $\partial B^4$ and carries $\Sigma_0$ to $\Sigma_1$. To that end, observe that $I_0$ and $H_0$ are the identity on $B^4$, so $I_0(H_0(\Sigma_0))=\Sigma_0$. To see that $I_1(H_1(\Sigma_0))$ equals $\Sigma_1$, it suffices to show that $I_1(\Sigma_1)=\Sigma_1$. In fact, since $h_1 = \id$, the diffeomorphism $I_1$ is the identity on $S^3 \times [-1,1]$ because 
$$
I_1(x,s) = \left( h^{-1}_1(x),s\right)= (x,s)
$$
for $s \in [-1,1]$. Since $\Sigma_1$ lies in  $S^3 \times [-1,1]$, we see that $I_1(H_1(\Sigma_0))=I_1(\Sigma_1)=\Sigma_1$.

Finally, to see that $I_t \circ H_t$ fixes $\partial B^4$, we plug in $s=1$:
$$I_t\left( H_t (x,1)\right) = I_t\left(h_t(x),1\right)=\left(h^{-1}_{t}(h_t(x)),1\right)=(x,1).$$
Thus we conclude that $\Sigma_0$ and $\Sigma_1$ are ambiently isotopic rel boundary.
\end{proof}

\subsection{Symmetry groups.}\label{subsec:asym} We  show that, for $m \gg0$, the knots $J_{m,n}$ from the proof of Theorem~\ref{thm:main} (and Examples~\ref{ex:asym}-\ref{ex:higher-genus}) are hyperbolic with trivial symmetry group. Consider the three-component link $L$ shown in Figure~\ref{fig:surgery}. Observe that the knot complement $S^3 \setminus J_{m,n}$ is obtained from $S^3 \setminus L$ by performing the indicated Dehn filling on the two unknotted link components. 

\begin{figure}[h]\center
\def\svgwidth{.385\linewidth}
\begingroup%
  \makeatletter%
  \providecommand\color[2][]{%
    \errmessage{(Inkscape) Color is used for the text in Inkscape, but the package 'color.sty' is not loaded}%
    \renewcommand\color[2][]{}%
  }%
  \providecommand\transparent[1]{%
    \errmessage{(Inkscape) Transparency is used (non-zero) for the text in Inkscape, but the package 'transparent.sty' is not loaded}%
    \renewcommand\transparent[1]{}%
  }%
  \providecommand\rotatebox[2]{#2}%
  \newcommand*\fsize{\dimexpr\f@size pt\relax}%
  \newcommand*\lineheight[1]{\fontsize{\fsize}{#1\fsize}\selectfont}%
  \ifx\svgwidth\undefined%
    \setlength{\unitlength}{241.92610373bp}%
    \ifx\svgscale\undefined%
      \relax%
    \else%
      \setlength{\unitlength}{\unitlength * \real{\svgscale}}%
    \fi%
  \else%
    \setlength{\unitlength}{\svgwidth}%
  \fi%
  \global\let\svgwidth\undefined%
  \global\let\svgscale\undefined%
  \makeatother%
  \begin{picture}(1,0.6687269)%
    \lineheight{1}%
    \setlength\tabcolsep{0pt}%
    \put(0,0){\includegraphics[width=\unitlength,page=1]{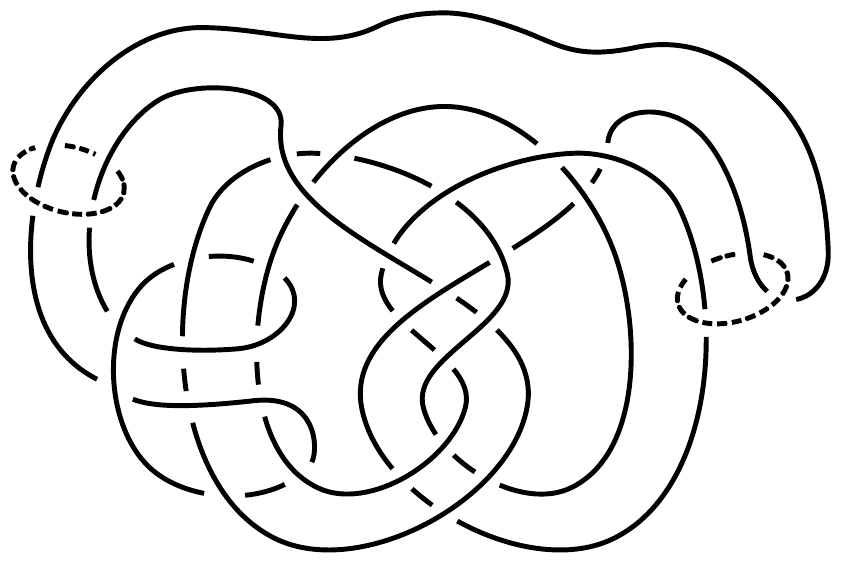}}%
    \put(-0.08081053,0.46122786){\makebox(0,0)[lt]{\lineheight{1.25}\smash{\begin{tabular}[t]{l}$\frac{1}{m}$\end{tabular}}}}%
    \put(0.88244112,0.22650448){\makebox(0,0)[lt]{\lineheight{1.25}\smash{\begin{tabular}[t]{l}$-\frac{1}{n}$\end{tabular}}}}%
  \end{picture}%
\endgroup%

\caption{A $3$-component link consisting of $J_{m,n}$ and two unknotted link components (dotted) on which we perform the indicated Dehn filling.}\label{fig:surgery}
\end{figure}

\begin{lem}
The link complement $S^3 \setminus L$ is hyperbolic  with trivial isometry group.
\end{lem}

\begin{proof}
We used SnapPy's link editor to obtain a Dowker-Thistlethwaite code for $L$:
\begin{align*}
\text{DT:}[(38,-32,-26,-56,-22,60,48,-6,34,-40,52,-12,58,16,-64,18,-4,30,\\-20,2,
-36,66,68,-8,24,-62,-70,46,14,-50,10),(-28,42),(-54,44)]
\end{align*}
Enter $L$ into Sage using \texttt{L=snappy.ManifoldHP(`DT:[(...)]')} and produce a triangulation of $S^3 \setminus L$  via \texttt{R=L.canonical\_retriangulation(verified=True)}.  Verify the isometry group is trivial via  \texttt{len(R.isomorphisms\_to(R))}, which returns \texttt{1}.
\end{proof}

\begin{prop}\label{prop:large}
For any $n \geq 0$ and sufficiently large $m \gg 0$,   the knot complement $S^3 \setminus J_{m,n}$ is hyperbolic with trivial isometry group.
\end{prop}

\begin{proof}
Thurston's hyperbolic Dehn surgery theorem \cite{thurston:notes}  implies that, when $|m|$ and $|n|$ are both sufficiently large, the Dehn-filled 3-manifold $S^3 \setminus J_{m,n}$ is hyperbolic and  the cores of the surgered solid tori are the unique shortest closed geodesics in $S^3 \setminus J_{m,n}$. In this case, any isometry of  $S^3 \setminus J_{m,n}$ must fix these short geodesics setwise, hence restricts to an isometry of their complements, i.e., $S^3 \setminus L$ (c.f., \cite[\S5]{kojima}). The latter has trivial isometry group, which in turn implies that $S^3 \setminus J_{m,n}$ has trivial isometry group when $|m|$ and $|n|$ are both sufficiently large.

We can use the work of Futer-Purcell-Schleimer \cite{futer-purcell-schleimer} to quantify these thresholds and extend the argument to all $n \geq 0$. As a preliminary step, we check that the systole length of $S^3 \setminus L$ is at least $0.1428$. To do so, we continue with the Sage session described above, entering \texttt{L.length\_spectrum(cutoff=1.0)} to  list all geodesics in $S^3 \setminus L$ of length at most 1. This returned a single curve, which had length approximately 0.977. Thus the systole length of $S^3 \setminus L$ is greater than $0.1428$.

Next, let $Y_n$ denote the 3-manifold obtained from $S^3 \setminus L$ by performing only the $-1/n$-filling along the unknotted component link component on the right. By \cite[Theorem~7.28]{futer-purcell-schleimer}, if the normalized length of the $-1/n$-filling slope is at least 10.1, then $Y_n$ is hyperbolic and the core of the surgered solid torus is the unique shortest closed geodesic in $Y_n$. By the same arguments as above, it will then follow that $Y_n$ is hyperbolic with trivial isometry group for all such $n$. 

To determine the normalized length of a slope $\gamma$, we fix a cusp $C$ on which to measure the length of $\gamma$, then normalize it as $\operatorname{Length}(\gamma)/\sqrt{\operatorname{Area}(C)}$. To measure these cusp areas, we use \texttt{L.cusp\_areas(verified=True)}, which tells us that the area of the cusp in question is approximately 10.745.  Next we ask SnapPy for all of the slopes on the given cusp that have length at most 34, which ensures a normalized length of
$$\frac{34}{\sqrt{\operatorname{Area}(C)}} > \frac{34}{\sqrt{11}} \approx 10.251 > 10.1,$$
as desired. We find these slopes using \texttt{L.short\_slopes(verified=True,length=34)}. SnapPy returns a list that includes the $-1/n$-slopes for $0 \leq n \leq 23$. Therefore, for $n > 23$, it follows that $Y_n$ is hyperbolic with trivial isometry group. For $0 \leq n \leq 23$, we run a loop to directly determine the hyperbolicity and symmetry groups of $Y_n$:
\begin{quote}
\tt
for n in range(0,24):

\qquad L.dehn\_fill((1,-n),2) 

\qquad Y = L.filled\_triangulation()

\qquad R = Y.canonical\_retriangulation(verified=True) 

\qquad Y = R.with\_hyperbolic\_structure()

\qquad Y.verify\_hyperbolicity()

\qquad len(R.isomorphisms\_to(R))
\end{quote}
For each value $n=0,1,\ldots,23$, this prints a confirmation that $Y_n$ is hyperbolic and that its isometry group is trivial.
\end{proof}

\vspace{-.45in}

\section{Reidemeister induced chain maps.} \label{subsec:chain_maps} In order to create a toolkit of explicitly defined Reidemeister induced chain maps, we list them all here. It is easier to record the Reidemeister III induced chain maps as tables, interpreted in the following manner. The top left cell of each table indicates the Reidemeister move: the bottom left corner of the cell gives the starting (local) diagram and the top right corner gives the ending. Similarly, the left column gives smoothings of the starting diagram, and the top row gives smoothings of the ending diagram. The row associated to a smoothing of the starting diagram defines the chain map on that smoothing (and any of its labelings): empty cells map to $0$; a cell with an $I$ maps to the corresponding smoothing in that column by an isotopy; a cell with a decoration of the smoothing is as expected (see Section \ref{subsec:induced_maps}). 

For convenience, we have also given an enumeration of the crossings and labeled the corresponding binary sequence for each smoothing (a different enumeration will not change the map). Also for convenience, we have listed two extra Reidemeister III moves: Tables \ref{table_r3_positive_under} and \ref{table_r3_negative_under} are determined from Tables \ref{table_r3_positive_over} and \ref{table_r3_negative_over}, respectively, by a rotation of the tangle (it is generally a headache to apply this rotation as well as the desired chain map).

\vspace{-0.5em}

\captionsetup{margin=0cm}

\begin{table}[!ht]
\bigskip
\begin{minipage}{.22\linewidth}
	\caption{Chain maps induced by a Reidemeister I.}
	\label{table_r1}
\end{minipage}
\quad
\begin{minipage}{.7\linewidth}
	\renewcommand{\arraystretch}{2.35}
	\begin{tabular}{|c|c|c|}
		\hline
		Reidemeister move & Smoothing & Induced map \\
		\hline
		\multirow{2}{80pt}{\begin{minipage}{9em}
				$\kropos \to \kroarc$
		\end{minipage}} & \krosa & \kroma $\begin{array}{c}\hspace{-1em}\vspace{.425em}\end{array}$ \\\cline{2-3} & \krosb & $\begin{array}{c} 0 \vspace{.425em} \end{array}$ \\
		\hline
		$\kroarc \to \kropos$ & \kroarc & $\begin{array}{c} \frac12 \Bigg( \kromb \Bigg) \vspace{.425em} \end{array}$ \\
		\hline
		\multirow{2}{80pt}{\begin{minipage}{9em}
				$\kroneg \to \kroarc$
		\end{minipage}} & \krosa & $\begin{array}{c} \frac12 \Bigg( \kromc \Bigg) \vspace{.425em} \end{array}$ \\\cline{2-3} & \krosb & $\begin{array}{c} 0 \vspace{.425em} \end{array}$ \\
		\hline
		$\kroarc \to \kroneg$ & \kroarc & \kromd $\begin{array}{c}\hspace{-1em}\vspace{.425em}\end{array}$ \\
		\hline
	\end{tabular}
	\renewcommand{\arraystretch}{1}
\end{minipage}
\end{table}

\vspace{0.5em}

\begin{table}[!ht] 
\begin{minipage}{.22\linewidth}
	\caption{Chain maps induced by a Reidemeister II.}
	\label{table_r2}
	\end{minipage}
	\quad
	\begin{minipage}{.7\linewidth}
		\renewcommand{\arraystretch}{2.35}
		\begin{tabular}{|c|c|c|}
			\hline
			Reidemeister move & Smoothing & Induced map \\
			\hline
			\multirow{2}{80pt}{\begin{minipage}{9em}
					$\krtcr \to \krtcrl$
			\end{minipage}} & \krtsa & $-$ \!\! \krtma $\begin{array}{c}\hspace{.25em}\vspace{.425em}\end{array}$ \\\cline{2-3} & \krtsb & \krtcrl $\begin{array}{c}\hspace{-1em}\vspace{.425em}\end{array}$ \\\cline{2-3} & \krtsc & $\begin{array}{c} 0 \vspace{.425em} \end{array}$ \\\cline{2-3} & \krtsd & $\begin{array}{c} 0 \vspace{.425em} \end{array}$ \\
			\hline
			$\krtcrl \to \krtcr$ & \krtsb & \krtmb $\begin{array}{c}\hspace{-1em}\vspace{.425em}\end{array}$ \\
			\hline
		\end{tabular}
		\hspace{1.75em}
		\renewcommand{\arraystretch}{1}
	\end{minipage}
\end{table}

\clearpage

\newgeometry{bottom=.6in}

\begin{table}[!ht]\center
	\hypertarget{table_r3}{}
\begin{minipage}{.18\linewidth}
	\renewcommand{\thetable}{5(a)}
	\caption{}
	\label{table_r3_positive_over}
	\end{minipage}
		\begin{minipage}{.72\linewidth}	\includegraphics[scale=.61]{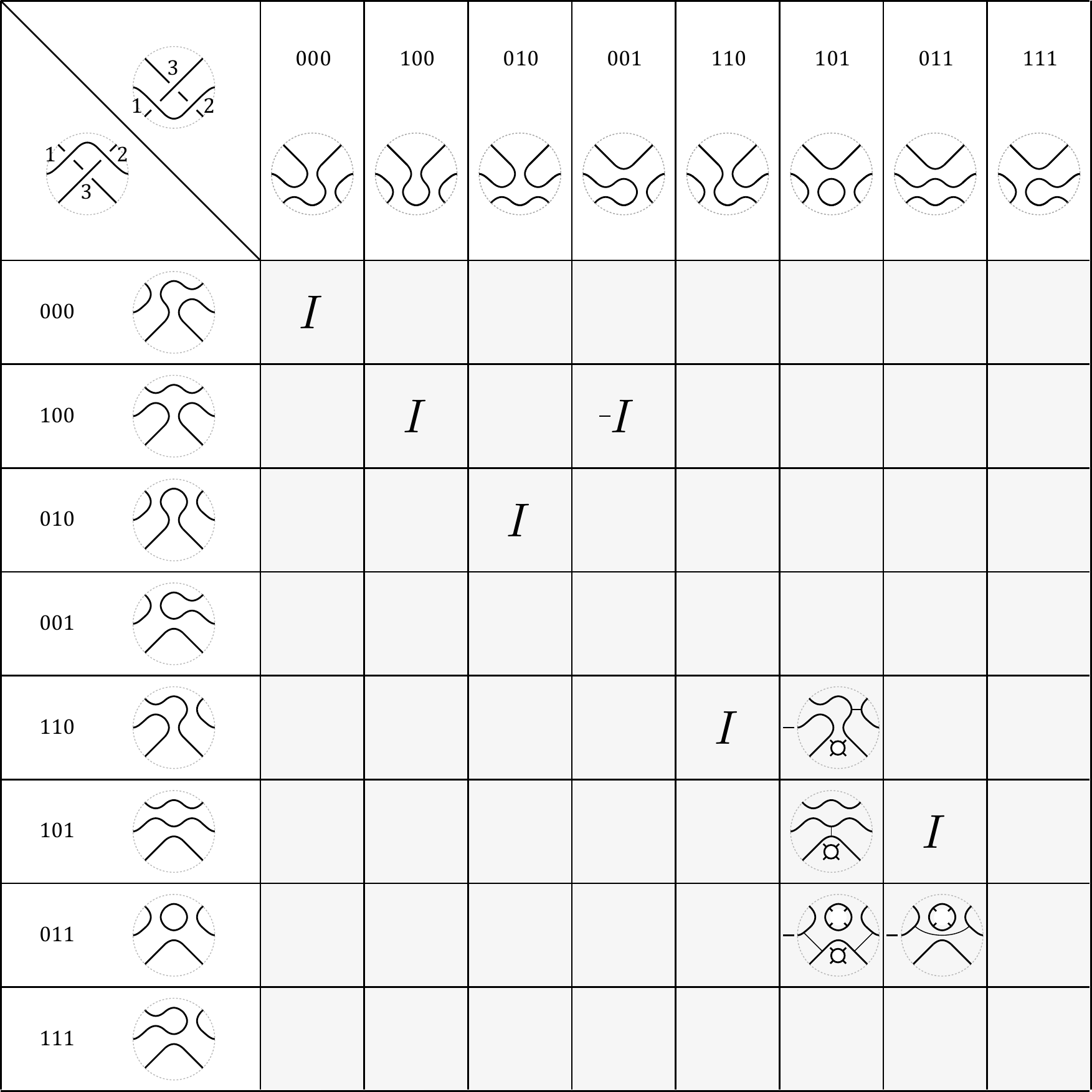}
	\end{minipage}
\end{table}

\begin{table}[!ht]\center
	\begin{minipage}{.18\linewidth}	
	\renewcommand{\thetable}{5(b)}
	\caption{}
	\label{table_r3_positive_under}
	\end{minipage}
	\begin{minipage}{.72\linewidth}	\includegraphics[scale=.61]{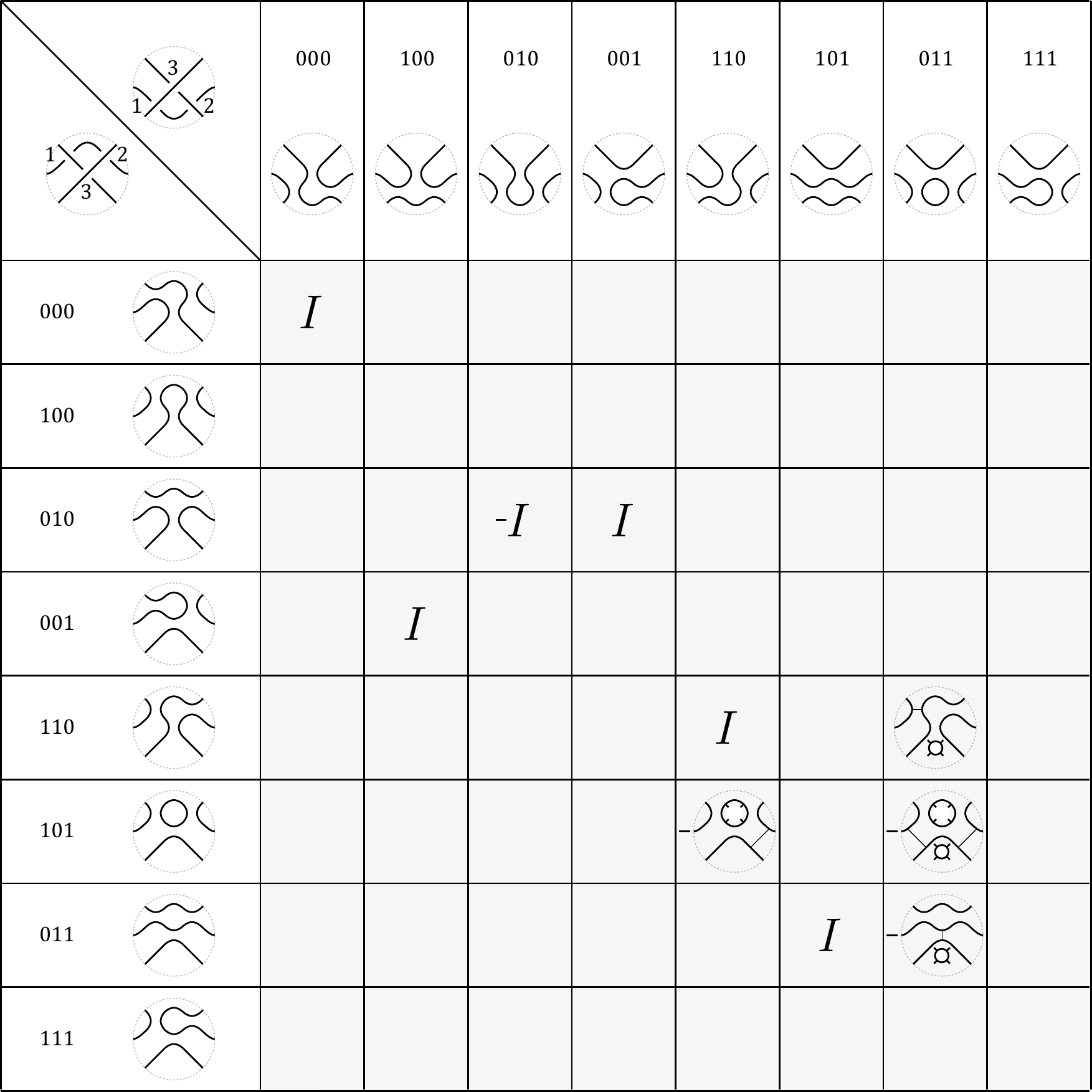}
	\end{minipage}
\end{table}

\newgeometry{bottom=.6in}

\begin{table}[!ht]\center
	\begin{minipage}{.18\linewidth}	
	\renewcommand{\thetable}{5(c)}
	\caption{}
	\label{table_r3_negative_over}
	\end{minipage}
	\begin{minipage}{.72\linewidth}	\includegraphics[scale=.61]{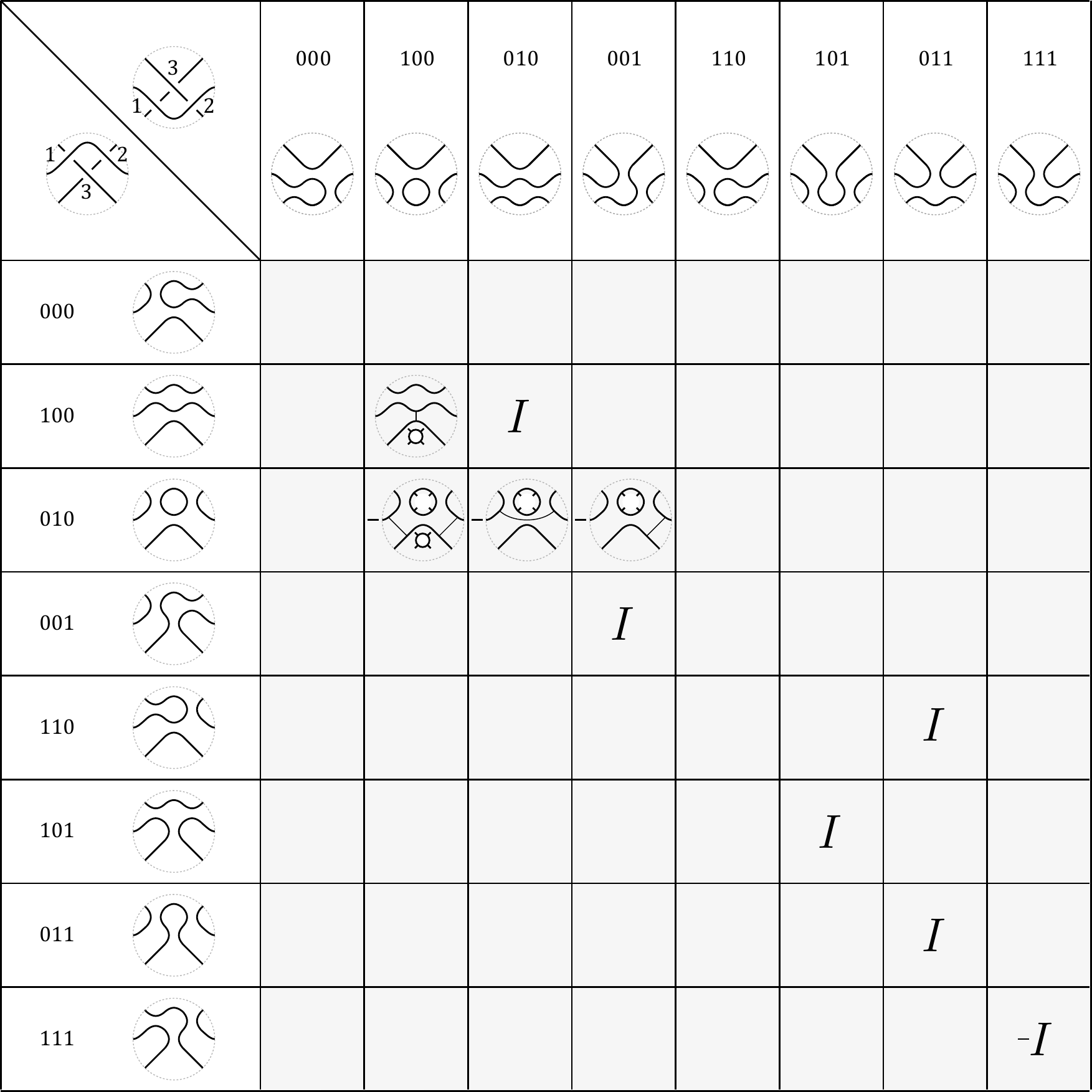}
	\end{minipage}
\end{table}

\begin{table}[!ht]\center
	\begin{minipage}{.18\linewidth}
		\renewcommand{\thetable}{5(d)}
         	\caption{}
	\label{table_r3_negative_under}
	\end{minipage}
	\begin{minipage}{.72\linewidth}	\includegraphics[scale=.61]{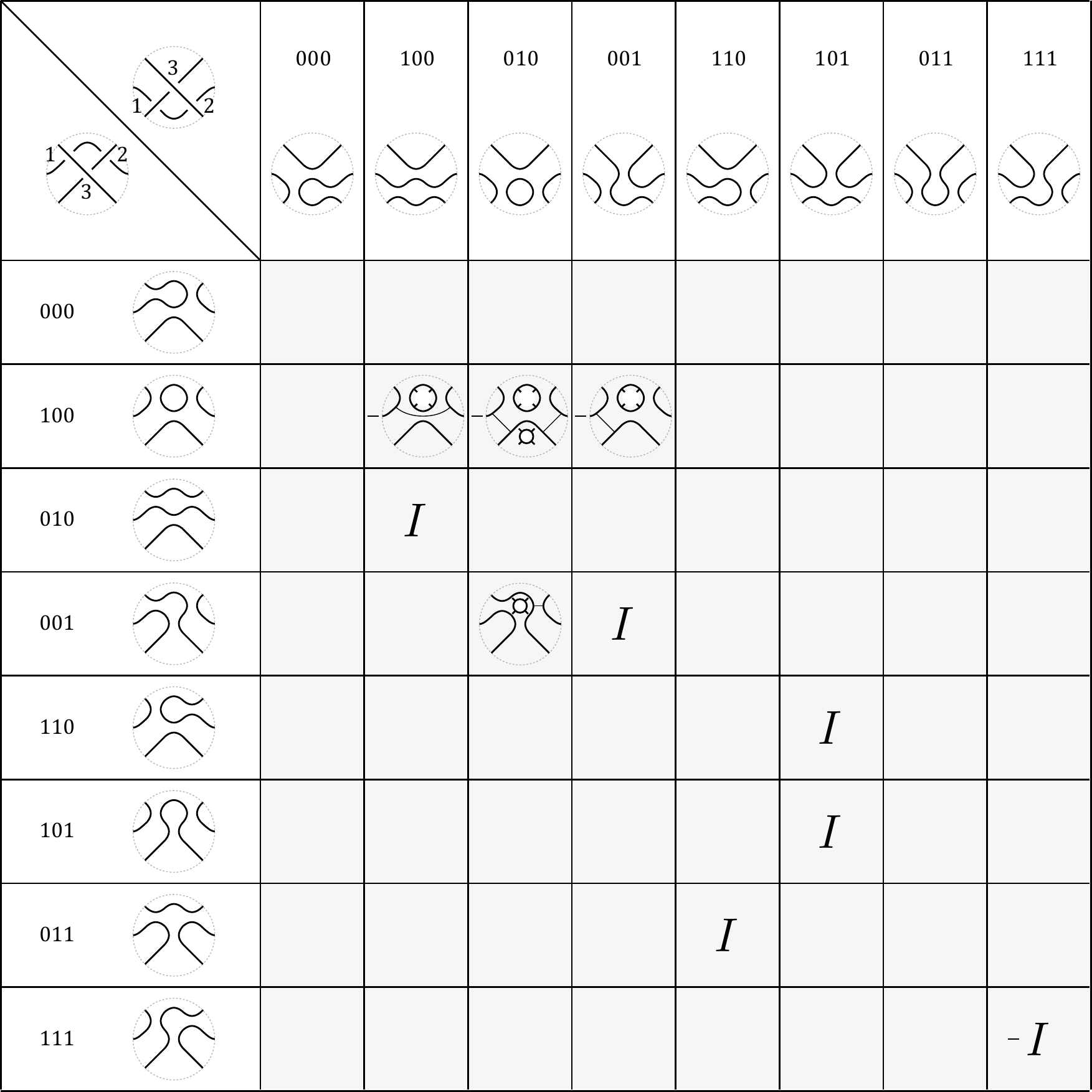}
	\end{minipage}
\end{table}

\clearpage


\newgeometry{margin=1.25in, left=1.45in, right=1.15in}


\titleformat{\section}
  {\normalfont\fontsize{12}{5}\bfseries}{}{\center 1em}{\center}

{\small \footnotesize \bibliographystyle{alphamod} \bibliography{main_tex}}

\end{document}